\documentclass[a4paper,11pt]{amsart}
\usepackage{geometry}
\usepackage{amsmath,amssymb,enumerate,bbm}
\usepackage{color}
\usepackage{graphicx}
\usepackage{url}
\usepackage{axodraw}
\usepackage[all]{xy}
\def\diagramme #1{\vskip 4mm \centerline {#1} \vskip 4mm}

\newcommand{\id}{\mathrm{id}}
 \font \eightrm=cmr8

\newcommand{\nc}{\newcommand}

\newcommand{\treeA}{\scalebox{0.5}{{\parbox{0.5pc}{
  \begin{picture}(30,45) (75,-60)
    \SetWidth{1.5}
    \SetColor{Black}
    \Vertex(90,-23){4}
    \Line(90,-30)(75,-60)
    \Line(90,-30)(105,-60)
    \Line(90,-15)(90,-30)
  \end{picture}}}}}

\newcommand{\treeB}{\scalebox{0.5}{{\parbox{0.5pc}{
  \begin{picture}(30,45) (75,-60)
    \SetWidth{1.5}
    \SetColor{Black}
    \Vertex(80,-50){4}
    \Line(90,-30)(75,-60)
    \Line(90,-30)(105,-60)
    \Line(90,-15)(90,-30)
  \end{picture}}}}}

\newcommand{\treeC}{\scalebox{0.5}{{\parbox{0.5pc}{
  \begin{picture}(30,45) (75,-60)
    \SetWidth{1.5}
    \SetColor{Black}
    \Vertex(100,-50){4}
    \Line(90,-30)(75,-60)
    \Line(90,-30)(105,-60)
    \Line(90,-15)(90,-30)
  \end{picture}}}}}

\newcommand{\treeD}{\scalebox{0.5}{{\parbox{0.5pc}{
  \begin{picture}(30,45) (75,-60)
    \SetWidth{1.5}
    \SetColor{Black}
    \Vertex(100,-50){4}
    \Vertex(80,-50){4}
    \Line(90,-30)(75,-60)
    \Line(90,-30)(105,-60)
    \Line(90,-15)(90,-30)
  \end{picture}}}}}

\nc{\ignore}[1]{{}}
\nc{\mrm}[1]{{\rm #1}}
\nc{\dirlim}{\displaystyle{\lim_{\longrightarrow}}\,}
\nc{\invlim}{\displaystyle{\lim_{\longleftarrow}}\,}
\nc{\vep}{\varepsilon} \nc{\ep}{\epsilon}
\nc{\sigmat}{\widetilde\sigma}
\nc{\ostar}{\overline{*}}

\nc{\mchar}{\mrm{Char}}
\nc{\Hom}{\mrm{Hom}}

\nc{\remark}{\noindent{\bf{Remark:}}}
\nc{\remarks}{\noindent{\bf{Remarks:}}}

 \nc{\delete}[1]{}
 \nc{\grad}[1]{^{({#1})}}
 \nc{\fil}[1]{_{#1}}

\nc{\BA}{{\Bbb A}} \nc{\CC}{{\Bbb C}} \nc{\DD}{{\Bbb D}}
\nc{\EE}{{\Bbb E}} \nc{\FF}{{\Bbb F}} \nc{\GG}{{\Bbb G}}
\nc{\HH}{{\Bbb H}} \nc{\LL}{{\Bbb L}} \nc{\NN}{{\Bbb N}}
\nc{\PP}{{\Bbb P}} \nc{\QQ}{{\Bbb Q}} \nc{\RR}{{\Bbb R}}
\nc{\TT}{{\Bbb T}} \nc{\VV}{{\Bbb V}} \nc{\ZZ}{{\Bbb Z}}
\nc{\Cal}[1]{{\mathcal {#1}}}
\nc{\mop}[1]{\mathop{\hbox {\rm #1} }\nolimits}
\nc{\smop}[1]{\mathop{\hbox {\eightrm #1} }\nolimits}
\nc{\mopl}[1]{\mathop{\hbox {\rm #1} }\limits}
\nc{\frakg}{{\frak g}}
\nc{\g}[1]{{\frak {#1}}}
\def \restr#1{\mathstrut_{\textstyle |}\raise-8pt\hbox{$\scriptstyle #1$}}
\def \srestr#1{\mathstrut_{\scriptstyle |}\hbox to
  -1.5pt{}\raise-4pt\hbox{$\scriptscriptstyle #1$}}
\nc{\wt}{\widetilde}
\nc{\wh}{\widehat}
\nc{\un}{\hbox{\bf 1}}
\nc{\redtext}[1]{\textcolor{red}{\tt #1}}
\nc{\bluetext}[1]{\textcolor{blue}{#1}}
\nc{\comment}[1]{[[{\tt {#1}}]] }
\nc{\R}{\mathbb R}
\nc\fleche[1]{\mathop{\hbox to #1 mm{\rightarrowfill}}\limits}
\def\semi{\mathrel{\times}\kern -.85pt\joinrel\mathrel{\raise
    1.4pt\hbox{${\scriptscriptstyle |}$}}}
\nc{\np}{/\hskip -2.3mm\pi}
\nc{\snp}{/\hskip -1.8mm\pi}
\def\shu{\joinrel{\!\scriptstyle\amalg\hskip -3.1pt\amalg}\,}
\def\sshu{\joinrel{\hskip 2pt\scriptscriptstyle\amalg\hskip -2.5pt\amalg}\,}

\def\shu{\joinrel{\!\scriptstyle\amalg\hskip -3.1pt\amalg}\,}
\def\qshu{\joinrel{\!\scriptstyle\amalg\hskip -3.1pt\amalg}\,\hskip -8pt\hbox{-}\hskip 5pt}
\def\sshu{\joinrel{\,\scriptscriptstyle\amalg\hskip -2.5pt\amalg}\,}
\def\sqshu{\joinrel{\,\scriptscriptstyle\amalg\hskip -2.5pt\amalg}\,\hskip -7pt\hbox{-}\hskip 4pt}
\def\sqshubis{\joinrel{\,\scriptscriptstyle\amalg\hskip -2.5pt\amalg}\,\hskip -8pt\hbox{-}\hskip 4pt}
\nc{\surj}{\to\hskip -3mm \to}

\newtheorem{thm}{Theorem}
\newtheorem{exam}{Example}

\newtheorem{cor}[thm]{Corollary}

\newtheorem{prop}[thm]{Proposition}

\newtheorem{rmk}[thm]{Remark}

\begin{document}

\title[Double shuffle structure of \lowercase{$q$}MZV\lowercase{s}]
	{Unfolding the double shuffle structure of \lowercase{$q$}MZV\lowercase{s}}

\author[J.~Castillo]{Jaime Castillo Medina}
\address{Univ.~de Val\`encia,
        	       Facultat de Ci\`encies Matem\`atiques,
	       C/ Doctor Moliner, 50, 46100 Burjassot-Valencia, Spain.}       
         \email{jaicasme@alumni.uv.es}

\author[K.~Ebrahimi-Fard]{Kurusch Ebrahimi-Fard}
\address{Inst.~de Ciencias Matem\'aticas,
		C/ Nicol\'as Cabrera, no.~13-15, 28049 Madrid, Spain.
		On leave from Univ.~de Haute Alsace, Mulhouse, France}
         \email{kurusch@icmat.es, kurusch.ebrahimi-fard@uha.fr}         
         \urladdr{www.icmat.es/kurusch}

\author[D.~Manchon]{Dominique Manchon}
\address{Univ. Blaise Pascal,
         C.N.R.S.-UMR 6620,
         63177 Aubi\`ere, France}       
         \email{manchon@math.univ-bpclermont.fr}
         \urladdr{http://math.univ-bpclermont.fr/~manchon/}

\date{17 December, 2013}         

\maketitle

\tableofcontents
 
\vspace{0.3cm}

\begin{abstract}
We exhibit the double $q$-shuffle structure for the \lowercase{$q$}MZVs introduced by Y.~Ohno, J.~Okuda and W.~Zudilin in the recent work \cite{OhOkZu}.
\end{abstract}

\noindent


\section{Introduction}
\label{sect:intro}

In \cite{OhOkZu} Y.~Ohno, J.~Okuda and W.~Zudilin proposed for positive natural numbers $n_1,n_2, \ldots, n_k \in \mathbb{N}$, and for any complex number $q$ with $|q|<1$, the following iterated infinite series:
\begin{equation}
\label{qMZV-z}
	\mathfrak{z}_q(n_1,\ldots,n_k)  	
	:=(1-q)^w\sum_{m_1 > \dots > m_k > 0}
			\frac{q^{m_1}}{(1-q^{m_1})^{n_1} \cdots (1-q^{m_k})^{n_k}} 						
\end{equation}
of depth $k \ge 1$ and weight $w := n_1 + \cdots + n_k$ as a $q$-analog of classical multiple zeta values (MZVs). Indeed, in the limit $q \to 1$ the $q$-number $[m]_q:=(1-q^m)(1-q)^{-1}$ becomes $m \in \mathbb{N}$, and the above infinite series reduces to the corresponding classical MZV of depth $k$ and weight $w$, defined for positive natural numbers $n_1, \ldots, n_k \in \mathbb{N}$ by:
\begin{equation}
\label{MZV}
	\zeta(n_1,\dots,n_k):=\sum_{m_1  > \dots > m_k > 0}\frac{1}{ m_1^{n_1}\cdots m_k^{n_k}}.     
\end{equation}
The extra condition $n_1>1$ is required in order that the right-hand side of \eqref{MZV} converges. The seminal works of M.~E.~Hoffman \cite{Hoffman1} and D.~Zagier \cite{Zagier} initiated a systematic study of MZVs. These real numbers appear in several contexts, e.g., arithmetic and algebraic geometry, algebra, combinatorics, mathematical physics, computer sciences, quantum groups and knot theory. See \cite{CartierMZV, Waldschmidt,Waldschmidt2} for introductory reviews. 

\medskip

M.~Kontsevich noticed that series (\ref{MZV}) has a simple representation in terms of iterated Riemann integrals \cite{Zagier}:   
\begin{equation} 
\label{ChenRepIntro}
	\zeta(n_1,\ldots, n_k) 
	= \idotsint\limits_{0\leq t_w \leq \cdots \leq t_1 \leq 1} 
			\frac{dt_1}{\tau_1(t_1)} \cdots \frac{dt_w}{\tau_w(t_w)},
\end{equation}
where $\tau_i(x) = 1-x$ if $i \in \{h_1,h_2,\ldots,h_k\}$, $h_j:=n_1+n_2+\cdots+n_j$, and $\tau_i(x) = x$ otherwise.

Multiplying MZVs represented in either form, i.e., (\ref{MZV}) or (\ref{ChenRepIntro}), results in $\mathbb{Q}$-linear combinations of MZVs. Hence, the $\mathbb{Q}$-vector space spanned by the real numbers (\ref{MZV}) forms an algebra. The weight of a product of MZVs is defined as the sum of the weights. The product of MZVs arising from the sum representation (\ref{MZV}) is usually called stuffle, or quasi-shuffle, product. It preserves the weight, but is not homogenous with respect to the depth. As an example, we look at the quasi-shuffle product of two MZVs of weight $a,b >1$, which is given by the Nielsen reflexion formula \cite{Waldschmidt2}:
$$
	\zeta(a)\zeta(b)=\zeta(a,b)+\zeta(b,a) + \zeta(a+b).
$$
Note that the last term is of depth one. The so-called shuffle product derives from integration by parts in the integral representation (\ref{ChenRepIntro}) of MZVs. It is homogenous with respect to both weight and depth. For natural numbers $a,b>1$, Euler's famous decomposition formula expresses the product of two single zeta values as:
\begin{equation}
\label{classical-euler}
	\zeta(a)\zeta(b)=
	\sum_{i=0}^{a-1} {i+b-1 \choose b-1} \zeta(b+i,a-i)
	+\sum_{j=0}^{b-1} {j+a-1 \choose a-1} \zeta(a+j,b-j).
\end{equation}
It is proven using integration by parts. As a result, from the interplay between the sum and the integral representations, products of MZVs can be written in two ways as $\mathbb{Q}$-linear combinations of MZVs. This leads to the so-called double shuffle relations among MZVs. The first example is given for $a=2=b$:
$$
	2\zeta(2,2) + \zeta(4) = 4\zeta(3,1) + 2\zeta(2,2),
$$
which implies $\zeta(4)=4\zeta(3,1)$. Let us also mention the regularization --or Hoffman's-- relations  \cite{IKZ}. The simplest one writes:
\begin{equation}
\label{euler}
	\zeta(2,1)=\zeta(3),
\end{equation}
and was already known by L.~Euler. This is just the tip of an iceberg. MZVs satisfy many relations over $\mathbb{Q}$, and this gives rise to rich algebro-combinatorial structures. The latter are described abstractly in terms of so-called shuffle and quasi-shuffle --Hopf--  algebras, which together encode the double shuffle relations for MZVs. We refer the reader to \cite{Hoffman2,Zudilin} for detailed reviews on the algebraic structures related to MZVs. 

\smallskip 

In a recent work \cite{CEM} we found a $q$-generalization of Euler's decomposition formula (\ref{classical-euler}) for so-called modified $q$-multiple zeta values ($q$MZVs):
\begin{equation}
\label{mod-qMZV-z}
	\bar{\mathfrak{z}}_q(n_1,\ldots,n_k)  	
	:=(1-q)^{-w}\mathfrak{z}_q(n_1,\ldots,n_k)
	 =\sum_{m_1 > \dots > m_k > 0} \frac{q^{m_1}}{(1-q^{m_1})^{n_1} \cdots (1-q^{m_k})^{n_k}} 					
\end{equation}
of depth $k \ge 1$ and weight $w := n_1+\cdots+n_k$ \cite{OhOkZu}. For $1 < a \leq b \in \mathbb{N}$ the product of two such series writes as:
\begin{eqnarray}
\label{Z-Euler}
	\bar{\mathfrak{z}}_q(a)\bar{\mathfrak{z}}_q(b) &=&
	\sum_{l=0}^{a-1}\ \sum_{k=0}^{a-1-l}  
		(-1)^k {l+b-1 \choose b-1}{b \choose k} \bar{\mathfrak{z}}_q(b+l,a-l-k) 						\\
	& & \hspace{0.5cm}
	 	+ \sum_{l=0}^{b-1}\ \sum_{k=0}^{\min(a,b-1-l)}  
		(-1)^k {l+a-1 \choose a-1}{a \choose k} \bar{\mathfrak{z}}_q(a+l,b-l-k) 				\nonumber\\
	& & \hspace{1cm}
		- \sum_{k=1}^{a} \beta_{a-k} \bar{\mathfrak{z}}_q(a+b-k)
			+ \sum_{j=1}^{a-1} \alpha_{a+b-1-j} \delta \bar{\mathfrak{z}}_q(a+b-1-j),		\nonumber
\end{eqnarray}
with coefficients $\beta_j,\alpha_k$ which both depend on the weights $a$ and $b$:
\begin{equation}
\label{coeff}
	\beta_j=  (-1)^{a-j}{j+b-1 \choose j,j+b-a,a-j-1} 
	\qquad\
	\alpha_k= \sum_{j=b}^{k} \frac{(-1)^{a+b-j}}{1-j}{j-1 \choose j-b,j-a,a+b-j-1}.
\end{equation}
Observe that the last summand on the right hand side of (\ref{Z-Euler}) contains the derivation $\delta:=q\frac{d}{dq}$. We shall therefore consider this identity in the differential algebra generated by (\ref{mod-qMZV-z}), with respect to the derivation $\delta$. One can show that by multiplying both sides by $(1-q)^{a+b}$, in the limit $q \nearrow 1$ formula (\ref{Z-Euler}) reduces to Euler's decomposition formula (\ref{classical-euler}).\\

Several $q$-analogs of MZVs have appeared in the literature. Especially the one proposed by Bradley in \cite{Bradley1}, extending the definition of  $q$-zeta function proposed by M. Kaneko, N. Kurokawa and M. Wakayama in \cite{KKW02}, has been studied in detail (see e.g.~\cite{OhOkZu,Zhao,Zudilin}). Bradley's $q$-analog of MZVs can be related to linear combinations of series (\ref{qMZV-z}). See \cite{CEM} for details. Another $q$-analog of MZVs has been recently discovered by H. Bachmann and U. K\"uhn \cite{BK13}: the algebra of multiple divisor functions. This version fits well with multiple Eisenstein series and modularity issues. Let us also mention the model proposed by K. G. Schlesinger \cite{Schlesinger}, in which the modified $q$-MZVs $\mathfrak z_q^S(n_1,\ldots,n_k)$ are defined by \eqref{qMZV-z} with $q^{m_1}$ replaced by $1$ in the numerator. Finally one can mention the recent work of A. Okounkov \cite{Okounkov}, where deep conjectures are proposed for yet another model, in which the numerators are based on palindromic polynomials.\\
 
In \cite{CEM}  we also showed the quasi-shuffle like product of modified $q$MZVs (\ref{mod-qMZV-z}) of weight $a, b>1$:
\begin{equation}
\label{mod-q-shuffle}
    	\bar{\mathfrak{z}}_q(a)\bar{\mathfrak{z}}_q(b) 
	= \bar{\mathfrak{z}}_q(a,b) + \bar{\mathfrak{z}}_q(b,a) + \bar{\mathfrak{z}}_q(a+b) 				
		- \bar{\mathfrak{z}}_q(a,b-1) - \bar{\mathfrak{z}}_q(b,a-1) - \bar{\mathfrak{z}}_q(a+b-1).
\end{equation}
We started to explore how the resulting double $q$-shuffle relations may be used to express the $\delta$-derivation terms in (\ref{Z-Euler}) as linear combinations of $q$MZVs. The simplest example yields for $a=2=b$:
$$
	\delta  \bar{\mathfrak{z}}_q(2) = 
	  4 \bar{\mathfrak{z}}_q(3,1) - 2 \bar{\mathfrak{z}}_q(2,1)
		-  \bar{\mathfrak{z}}_q(2)
		+ 3 \bar{\mathfrak{z}}_q(3) 
	 	- \bar{\mathfrak{z}}_q(4). 
$$

\medskip

In this paper we shall complement the results in \cite{CEM} by unfolding the double $q$-shuffle structure for the modified $q$MZVs defined in (\ref{mod-qMZV-z}). Contrarily to what happens in the classical $q=1$ case, these $q$-multiple zeta values $\mathfrak{z}_q(n_1,\ldots,n_k)$ and $\bar{\mathfrak{z}}_q(n_1,\ldots,n_k)$ make sense for any $n_1,\ldots,n_k\in{\mathbb Z}$, regardless of the sign. Hence no further regularization issues arise in this context. The $q$-quasi-shuffle structure will then take place in terms of words with letters in an alphabet indexed by $\mathbb Z$, and the $q$-shuffle structure will be explained in terms of a particular \textsl{weight $-1$ differential Rota--Baxter algebra} \cite{GK} with invertible Rota--Baxter operator.  We will use these double $q$-shuffle relations to derive an expression for $\delta$-terms in terms of modified $q$MZVs.

\medskip

A double shuffle picture for $q$-multiple zeta values has also been indicated recently in the Bradley model by Y.~Takeyama \cite{Takeyama}, for positive arguments $n_1,\ldots,n_k$.  Both approaches to the $q$-shuffle relations use a representation of $q$-multiple zeta values by a $q$-analogue of multiple polylogarithms with one variable. This representation is less direct in the Bradley model than in the Ohno--Okuda--Zudilin model, and needs a ``twisting" of the words (see Lemma 2.10 loc.~cit.) which is not needed in our approach.

\smallskip

The paper is organized as follows: after a quick review of the double shuffle structure for classical MZVs in Section \ref{sect:algMZVs}, we introduce in Section \ref{sect:Jackson} the Jackson integral $J$, the $q$-difference operator $D_q$ defined by $D_q[f](t):=f(t)-f(qt)$ and the $q$-summation operator $P_q$ defined by $P_q[f](t):=f(t)+f(qt)+f(q^2t)+\cdots$. Operators $P_q$ and $D_q$ are mutually inverse in a suitable space of formal series. For any $a,b \in \mathbb Z$, explicit expressions for products $P_q^a[f]P_q^b[g]$ are given in terms of $P_q^i\big[P_q^j[f]g\big]$, $P_q^i\big[fP_q^j[g]\big]$ and $P_q^i[fg]$ for some $i,j \in \mathbb Z$.

\medskip

In Section \ref{sect:qMZV} we show how the modified Ohno--Okuda--Zudilin $q$MZVs can be expressed in terms of $q$-summation and $q$-difference operators, namely:
\begin{equation*}
	\bar{\mathfrak{z}}_q(n_1,\ldots,n_k) = P_q^{n_1}\big[\bar{y} P_q^{n_2}[\bar y\cdots P_q^{n_k}[\bar{y}]\cdots] \big](q),
\end{equation*}
with $\bar y(t):=\frac{t}{1-t}$, and we give a complete picture of the double shuffle structure. The $q$-quasi-shuffle algebra is built on the space of words on the alphabet $\wt Y:=\{z_n,\,n\in\mathbb Z\}$. The $q$-quasi-shuffle product $\qshu$ is the ordinary quasi-shuffle product, twisted in a certain sense by the operator $T: z_j \mapsto z_j-z_{j-1}$. The $q$-shuffle algebra consists of words with letters from the alphabet $\wt X:=\{d,y,p\}$, all ending with $y$, and subject to $dp=pd=\un$, so that one can also use the notation $p^{-1}=d$. The $q$-shuffle product is given recursively (with respect to the length of words) by:
\begin{eqnarray*}
	(yv)\shu u=v\shu (yu)	 &=&	 y(v\shu u),\\
				dv\shu du  &=& v\shu du + dv\shu u - d(v\shu u),\\
		  		  pv\shu pu&=& p(v\shu pu) + p(pv\shu u) - p(v\shu u),\\
		dv\shu pu=pu\shu dv &=& d(v\shu pu) + dv\shu u - v\shu u
\end{eqnarray*}
for any words $v$ and $u$. We show that the product $\shu$ is commutative and associative. Let us mention that it would be nice to have a purely combinatorial description of this product. The double shuffle relations then simply write:
\begin{eqnarray*}
	\bar{\mathfrak{z}}_q^{\sshu}(v)\bar{\mathfrak{z}}_q^{\sshu}(u)	 
	&=&\bar{\mathfrak{z}}_q^{\sshu}(v\shu u),\\
	\bar{\mathfrak{z}}_q^{\sqshu}(v')\bar{\mathfrak{z}}_q^{\sqshu}(u')
	&=&\bar{\mathfrak{z}}_q^{\sqshu}(v'\qshu u')
\end{eqnarray*}
for any words $v,u$ (respectively~ $v',u'$) of letters form the alphabet $\wt X$ (respectively~$\wt Y$), with:
\begin{equation*}
	\bar{\mathfrak{z}}_q(n_1,\ldots,n_k)	
						= \bar{\mathfrak{z}}_q^{\sshu}(p^{n_1}y\cdots p^{n_k}y)
						= \bar{\mathfrak{z}}_q^{\sqshu}(z_{n_1}\cdots z_{n_k}).
\end{equation*}
Finally we give an explicit expression of  $\delta \bar{\mathfrak{z}}_q(n_1,\ldots,n_k)$ for any $n_1,\ldots n_k\in\mathbb Z$, and state a $q$-analog of Hoffman's regularization relations. Note that we have not explored possible Hopf algebra structures at this stage, and that the limit for $q\to 1$ still needs to be studied in detail. We plan to address these issues in a future work.

\vspace{0.5cm}
{\bf{Acknowledgements}}: We thank W.~Zudilin for his most valuable comments and encouragements, as well as H.~Bachmann and U.~K\"uhn for illuminating discussions. The first author gratefully acknowledges support by the ICMAT Severo Ochoa Excellence Program. He would like to thank ICMAT for warm hospitality during his visit, and the third author as well. The second author is supported by a Ram\'on y Cajal research grant from the Spanish government. The second and third authors are supported by CNRS (GDR Renormalisation).\\[.3cm]


\section{Regularized double shuffle relations for MZVs}
\label{sect:algMZVs}

Let us briefly recall the double shuffle algebra for MZVs. The reader looking for details may consult the references \cite{Hoffman2,Waldschmidt,Zudilin}.

In the introduction we have seen that MZVs are either represented by iterated sums (\ref{MZV}) or in terms of iterated integrals (\ref{ChenRepIntro}). Thus it is convenient to write them in terms of words. In view of \eqref{MZV} and \eqref{ChenRepIntro}, this can be done by using two alphabets:
\begin{equation*}
	X:=\{x_0,x_1\},\hskip 12mm Y:=\{y_1,y_2,y_3,\ldots\}.
\end{equation*}
We denote by $X^*$ respectively $Y^*$ the set of words with letters in $X$ respectively $Y$. The length of words in $X^*$ and $Y^*$ is defined by the number of letters, and is denoted by $|u|$ for $u$ in $X^*$ or $Y^*$. For both $X^*$ and $Y^*$ we denote the empty word by $\un$.

The vector space ${\mathbb Q}\langle X\rangle$, which is freely generated by $X^*$, is a commutative algebra for the shuffle product, which is defined by:
\begin{equation}
\label{shu-prod}
	v_1\cdots v_p\shu v_{p+1}\cdots v_{p+q}:=
		\sum_{\sigma\in\smop{Sh}(p,q)}v_{\sigma_1^{-1}}\cdots v_{\sigma_{p+q}^{-1}}
\end{equation}
with $v_j\in X$, $j\in\{1,\ldots ,p+q\}$. Here, $\mop{Sh}(p,q)$ is the set of $(p,q$)-shuffles, i.e., permutations $\sigma$ of $\{1,\ldots, p+q\}$ such that $\sigma_1<\cdots < \sigma_p$ and $\sigma_{p+1}<\cdots <\sigma_{p+q}$. We define $\un \shu v = v \shu \un =v$ for $v \in X^*$. The product is homogenous with respect to the length of words. 

The vector space ${\mathbb Q}\langle Y\rangle$ is freely generated by $Y^*$. It is a commutative algebra for the quasi-shuffle product, which is defined as follows: a $(p,q)$-quasi-shuffle of type $r$ is a surjection: 
$$
	\sigma:\{1,\ldots, p+q\}\surj\{1,\ldots, p+q-r\},
$$
such that  $\sigma_1<\cdots < \sigma_p$ and $\sigma_{p+1}<\cdots <\sigma_{p+q}$. Denoting by $\mop{Qsh}(p,q;r)$ the set of $(p,q)$-quasi-shuffles of type $r$, the formula for the quasi-shuffle product $\qshu$ is:
\begin{equation}
\label{qshu-prod}
	u_1\cdots u_p \qshu u_{p+1}\cdots u_{p+q} :=
		\sum_{r\ge 0}\,\sum_{\sigma\in\smop{Qsh}(p,q;r)}u_1^\sigma\cdots u_{p+q-r}^\sigma
\end{equation}
with $u_j \in Y$, $j\in\{1,\ldots ,p+q\}$, and where $u_j^{\sigma}$ is the \textsl{internal product} of letters in the set $\sigma^{-1}(\{j\})$, which contains one or two elements. This internal product on $Y$ is defined by $[y_ky_l]:=y_{k+l}$. As before, we define $\un \qshu u = u \qshu \un =u$ for $u \in Y^*$.  We notice that the right hand side of (\ref{qshu-prod}) is not homogenous with respect to the length of words.\\

Both (\ref{shu-prod}) and (\ref{qshu-prod}) can be defined recursively. Indeed, for the shuffle product one can show that:
\begin{equation}
\label{shu-prod-rec}
	v_1\cdots v_p \shu v_{p+1}\cdots v_{p+q} = v_1\big( v_2\cdots v_p \shu v_{p+1}\cdots v_{p+q}  \big) 
										+ v_{p+1}\big( v_1\cdots v_p \shu v_{p+2}\cdots v_{p+q} \big). 
\end{equation}
The quasi-shuffle product is defined recursively by:
\begin{eqnarray}
\label{qshu-prod-rec}
	(u_1\cdots u_p)\qshu(u_{p+1}\cdots u_{p+q}) &=& u_1\big( u_2\cdots u_p \qshu u_{p+1}\cdots u_{p+q}  \big) \\
							& &	\hspace{1cm}+ u_{p+1}\big(  u_1\cdots u_p \qshu u_{p+2}\cdots u_{p+q} \big)\nonumber\\
							& &	\hspace{2.5cm}	+ [u_1u_{p+1}]\big(  u_2\cdots u_p \qshu u_{p+2}\cdots u_{p+q} \big).
												 \nonumber
\end{eqnarray}							  

We denote by $Y^*_{\smop{conv}}$ the submonoid of words $u=u_1\cdots u_r$ with $u_1\not =y_1$, and we set $X^*_{\smop{conv}}:=x_0X^*x_1$. An injective monoid morphism is given by changing the letter $y_n$ into the word $x_0^{n-1}x_1$, namely:
\begin{eqnarray*}
	{\frak s}:Y^*&\longrightarrow & X^*\\
	y_{n_1}\cdots y_{n_r}&\longmapsto & x_0^{n_1-1}x_1\cdots  x_0^{n_r-1}x_1.
\end{eqnarray*}
It restricts to a monoid isomorphism from $Y^*_{\smop{conv}}$ onto $X^*_{\smop{conv}}$. As notation suggests,  $Y^*_{\smop{conv}}$ and $X^*_{\smop{conv}}$ are two convenient ways to symbolize convergent MZVs through representations \eqref{MZV} and \eqref{ChenRepIntro}, respectively. Indeed, we can define:
$$
	\int_1(x_0^{n_1-1}x_1\cdots  x_0^{n_k-1}x_1):= \idotsint\limits_{0\leq t_w \leq \cdots \leq t_1 \leq 1} 
			\frac{dt_1}{\tau_1(t_1)} \cdots \frac{dt_w}{\tau_w(t_w)},
$$
where $\tau_i(s) = 1-s$ if $i \in \{h_1,h_2,\ldots,h_k\}$, $h_j:=n_1+n_2+\cdots+n_j$, and $\tau_i(s) = s$ otherwise. For the iterated sums we define:
$$
	\Sigma(y_{n_1}\cdots y_{n_k}):= \sum_{m_1  > \dots > m_k > 0}\frac{1}{ m_1^{n_1}\cdots m_k^{n_k}}.     
$$

The following notation is commonly adopted:
\begin{eqnarray*}
	\zeta_{\sqshu}(y_{n_1}\cdots y_{n_r})&:=&\Sigma(y_{n_1}\cdots y_{n_r})\\
								&=&\zeta(n_1,\ldots, n_r)\\
								&=:&\int_1(x_0^{n_1-1}x_1\cdots  x_0^{n_r-1}x_1)
								=\zeta_{\sshu}\big({\frak s}(y_{n_1}\cdots y_{n_r})\big),
\end{eqnarray*}
and extended to finite linear combinations of convergent words by linearity. The relation:
 $$
 	\zeta_{\sqshu}=\zeta_{\sshu}\circ{\frak s}
$$
is obviously verified. The map $\Sigma$ defines an algebra homomorphism, i.e., the quasi-shuffle relations then write:
\begin{equation}
\label{qsh-rel}
	\zeta_{\sqshu}(u\qshu u')=\zeta_{\sqshu}(u)\zeta_{\sqshu}(u')
\end{equation}
for any $u,u'\in Y^*_{\smop{conv}}$. The map $\int_1$ defines another algebra homomorphism, and the shuffle relations write:
\begin{equation}
\label{sh-rel}
	\zeta_{\sshu}(v\shu v')=\zeta_{\sshu}(v)\zeta_{\sshu}(v')
\end{equation}
for any $v,v'\in X^*_{\smop{conv}}$. 

By fixing an arbitrary value $\theta$ to $\zeta(1)$ and setting $\zeta_{\sqshu}(y_1)=\zeta_{\sshu}(x_1)=\theta$, it is possible to extend $\zeta_{\sqshu}$, respectively $\zeta_{\sshu}$, to all words in $Y^*$, respectively to $X^*x_1$, such that \eqref{qsh-rel}, respectively \eqref{sh-rel}, still hold. It is also possible to extend $\zeta_{\sshu}$ to a map defined on $X^*$ by fixing an arbitrary value $\theta'$ to $\zeta_{\sshu}(x_0)$, such that \eqref{sh-rel} is still valid. It is natural to suppose $\theta'=\theta$ for symmetry reasons, reflecting the following formal equality between two infinite quantities:
$$
	\int_0^1\frac{dt}{t}=\int_0^1\frac{dt}{1-t}.
$$
It is easy to show that for any words $v\in X^*$ or $u\in Y^*$, the expressions $\zeta_{\sshu}(v)$ and $\zeta_{\sqshu}(u)$ are polynomials with respect to $\theta$. It is no longer true that the extended $\zeta_{\sqshu}$ coincides with the extended $\zeta_{\sshu}\circ{\frak s}$, but the defect can be explicitly written:

\begin{thm}[L. Boutet de Monvel, D. Zagier \cite{Zagier}]
\label{bmz}
There exists an infinite-order invertible differential operator $\rho:\R[\theta]\to\R[\theta]$ such that
\begin{equation}
	\zeta_{\sshu}\circ{\frak s}=\rho\circ\zeta_{\sqshubis}.
\end{equation}
The operator $\rho$ is explicitly given by the series:
\begin{equation}
	\rho=\exp\left(\sum_{n\ge 2}\frac{(-1)^n\zeta(n)}{n}\left(\frac{d}{d\theta}\right)^n\right).
\end{equation}
\end{thm}

\noindent
In particular, $\rho(1)=1$, $\rho(\theta)=\theta$, and more generally $\rho(L)-L$ is a polynomial of degree $\le d-2$ if $L$ is of degree $d$, hence $\rho$ is invertible. A proof of Theorem \ref{bmz} can be read in numerous references, e.g. \cite{CartierMZV,IKZ,Racinet}. Any word $u \in Y^*_{\smop{conv}}$ gives rise to Hoffman's regularization relation:
\begin{equation}
\label{reg}
	\zeta_{\sshu}\big(x_1\shu {\frak s}(u)-{\frak s}(y_1\qshu u)\big)=0,
\end{equation}
which is a direct consequence of Theorem \ref{bmz}. The linear combination of words involved above is convergent, hence \eqref{reg} is a relation between convergent MZVs, although divergent ones have been used to establish it. The simplest regularization relation $\zeta(2,1)=\zeta(3)$ is obtained by applying  \eqref{reg} to the word $u=y_2$.


\section{The Jackson Integral}
\label{sect:Jackson}

The results in \cite{CEM} are essentially based on replacing the classical indefinite Riemann integral $R(f)(t):=\int_{0}^{t}f(y)\:dy$ in \eqref{ChenRepIntro} by its $q$-analog, known as Jackson's integral:
\begin{equation}
\label{JI}
    	J[f](t) := (1-q)\:\sum_{n \ge 0} f(q^nt) q^nt.  	
\end{equation}
G.C.~Rota \cite{Rota1,Rota2}  provided an elementary algebraic description of the map $J$ in terms of the $q$-dilation operator:
\begin{equation*}
    	E_q[f](t):=f(qt),         
\end{equation*}
together with  the multiplication operator $M_{f}[g](t):=(fg)(t)=f(t)g(t)$, such that:         
\begin{eqnarray}
\label{qInt}	
     J[f](t) = (1-q)\sum_{n \ge 0} f(q^nt) q^nt             	
             	  &=& (1-q)\sum_{n \ge 0} E^n_q\big[M_{\id}[f]\big](t) \\  
             	  &=:& (1-q)P_qM_{\id}[f](t).                   \nonumber	
\end{eqnarray}
One can recover the ordinary Riemann integral as the limit of the Jackson integral for $q\nearrow 1$. Note however that the Jackson integral makes sense purely algebraically if $q$ is considered as an indeterminate: more precisely, let $\Cal A=t{\mathbb Q}[[t,q]]$ be the space of formal series in two variables with strictly positive valuation in $t$. We can see $\Cal A$ as the space of series in $t$ without constant term and with coefficients in ${\mathbb Q}[[q]]$. Then \eqref{JI} defines the Jackson integral as a ${\mathbb Q}[[q]]$-linear endomorphism of $\Cal A$. 

The ${\mathbb Q}[[q]]$-linear map $P_q:\Cal A\to\Cal A$ defined by:
\begin{equation*}
    		P_q[f] (t):= \sum_{n \ge 0}E_q^{n}[f] = f(t) + f(qt) + f(q^2t) + f(q^3t) + \cdots       
\end{equation*}
satisfies the Rota--Baxter identity of weight $-1$:
\begin{equation}
\label{qRBR}
	P_q[f]P_q[g] = P_q\big[P_q[f]g\big] + P_q\big[fP_q[g]\big]  - P_q[fg \big]. 
\end{equation}
We refer the reader to \cite{CEM} and \cite{EFP,Rota1,Rota2} for more details regarding Rota--Baxter algebras and related topics. Moreover, to avoid confusion, we ask the reader to note that in \cite{CEM} we denoted the map $P_q$ by $\tilde{P}_q$. The former notation, however, is more appropriate for what follows.\\

From this it follows that Jackson's integral (\ref{qInt}) satisfies the relation:
\begin{equation*}
    	J[f] J[g] = J \big[J [f] \: g + f \: J[g] - (1-q)\id f g\big].
\end{equation*}
Or equivalently: 
$$
	J[f]J[g] =J\big[ fJ[g] \big]+  qJ\big[ J \big[E_q[f] \big]g \big],
$$
which is commonly considered to be the $q$-analog for the classical integration by parts rule: 
\begin{equation*}
	R(f)R(g)=R\big(fR(g)+ R(f)g\big), 
\end{equation*}
which can be seen as dual to Leibniz' rule for usual derivations. 

The inverse of the map $P_q$ is defined in terms of the finite $q$-difference operator:
\begin{equation*}
	D_q := I - E_q,
\end{equation*}
where $I$ is the identity map, $I[f]:=f$. Indeed, one shows quickly that $P_qD_q[f] = f=D_qP_q[f]$. It is easy to see that $D_q$ satisfies the generalized Leibniz rule for finite differences:
\begin{equation}
\label{qLeibniz}
	D_q[fg] = D_q[f]g + fD_q[g] - D_q[f]D_q[g].
\end{equation}
This makes $\Cal A$ a \textsl{weight $-1$ differential Rota--Baxter algebra} \cite{GK}, with the crucial additional property that $P_q$ and $D_q$ are mutually inverse\footnote{In a differential Rota--Baxter algebra with differential $D$ and Rota--Baxter operator $P$, equality $D\circ P=\mop{Id}$ holds but $P\circ D=\mop{Id}$ does not hold in general.}. Interestingly enough, \eqref{qLeibniz} can be reordered:
\begin{equation}
\label{qLeibniz-bis}
	D_q[f]D_q[g]= D_q[f]g +fD_q[g]-D_q[fg],
\end{equation}
thus showing a striking similarity with (\ref{qRBR}). This will play a crucial role in the sequel of this paper. \\

For positive $a,b \in \mathbb{N}$ the generalized Leibniz rule (\ref{qLeibniz-bis}) leads to the recursion:
$$
	D_q^a[f]D_q^b[g] = D_q^{a-1}[f]D_q^b[g] + D_q^a[f]D_q^{b-1}[g] - D_q\big[D_q^{a-1}[f]D_q^{b-1}[g]\big].
$$
From this we deduce the following theorem:

\begin{thm}
\label{thm:EulerDD}
Let $1<a\leq b \in \mathbb{N}$.  
\begin{eqnarray} 
\label{DD-Euler}
	D_q^a[f]D_q^b[g] 
	&=& \sum_{j=0}^{a-1} \sum_{i=1}^{b-j} (-1)^{j}
					{a+b-1-i-j \choose j, a-1-j,b-i-j}D_q^j\big[fD^i_q[g]\big]		\\
	& & \hspace{0.3cm} + \sum_{j=1}^{a} \sum_{i=1}^{b-j} (-1)^{j}
					{a+b-1-i-j \choose j-1, a-j,b-i-j}D_q^j\big[fD^i_q[g]\big]		 	\nonumber \\
	& & \hspace{0.5cm} + \sum_{j=0}^{b-1} \sum_{i=1}^{a-j} (-1)^{j}
					{a+b-1-i-j \choose j, b-1-j,a-i-j}D_q^j\big[D^i_q[f]g\big]			\nonumber \\			
	& & \hspace{0.7cm} +  \sum_{j=1}^{b} \sum_{i=1}^{a-j} (-1)^{j}
					{a+b-1-i-j \choose j-1, b-j,a-i-j}D_q^j\big[D^i_q[f]g\big]			\nonumber \\   
	& & \hspace{1.2cm}+  \sum_{j=1}^{a}(-1)^{j} {a+b-1-j \choose j-1,a-j,b-j} 
					D_q^j[fg].						\nonumber
\end{eqnarray} 	
\end{thm}

\begin{proof}
The proof of (\ref{DD-Euler})  follows the same lines of argument as given in \cite{DP}. See also \cite{CEM}. We briefly recall the basic idea \cite{DP}.  Identity (\ref{qLeibniz-bis}) can be represented pictorially:
$$
	\treeD \quad\ =\ \treeB \quad +\ \treeC \quad - \treeA
$$
where the dots represent the operator $D$. The branching indicates multiplication -- we skipped the decoration of the left and right leave by $f$ and $g$, respectively. Observe that each of the trees on the right hand side has one dot less than the tree on the left hand side. 

We define $\Gamma(a,b,c;f,g):=D^c_q(D^a_q(f)D^b_q(g))$, which would be represented by a tree with $a$ dots on the left branch, $b$ dots on the right brach, and $c$ dots on the upper branch. Using (\ref{qLeibniz-bis}) we would like to eliminate the dots on the lower branches, that is, we would like to write    $\Gamma(a,b,c;f,g)$ as a sum of   the three terms $\Gamma(0,i,c+j;f,g)$,  $\Gamma(i,0,c+j;f,g)$, and $\Gamma(0,0,c+j;f,g)$. 

Hence, starting with $a$ dots on the left branch, $b$ dots on the right branch, and $c$ dots on the upper branch, the authors in \cite{DP} describe the counting of possibilities of reducing the number of dots on either of the lower branches in terms of ``moving" dots successively upwards. Identity (\ref{qLeibniz-bis}) implies essentially  three moves, two of which eliminate a dot on either of the lower branches. The last move merges a dot from each brach into a new dot, which is then lifted upwards. The coefficients in (\ref{DD-Euler}) result from carefully counting the moves needed to get to either of the three terms $\Gamma(0,i,c+j;f,g)$,  $\Gamma(i,0,c+j;f,g)$ and  $\Gamma(0,0,c+j;f,g)$.             
\end{proof}

Note that the above expression can be simplified to the following more handy relation.

\begin{prop}
\label{prop:EulerDD}
Let $1<a\leq b \in \mathbb{N}$.  
\begin{eqnarray} 
\label{DD-Euler-slim}
	D_q^a[f]D_q^b[g] &=& 
	\sum_{j=0}^{a}\ \sum_{i=1}^{b-j}  
		(-1)^j {a+b-1-i-j \choose a-1}{a \choose j} D_q^j\big[fD^i_q[g]\big]	 					\\
	& & \hspace{0.5cm}
	 	+ \sum_{j=0}^{b}\ \sum_{i=1}^{\max(1,a-j)}  
		(-1)^j {a+b-1-i-j \choose b-1}{b \choose j} D_q^j\big[D^i_q[f]g\big]	 				\nonumber\\
	& & \hspace{1cm}
		+ \sum_{j=1}^{a} (-1)^j  {a+b-1-j \choose j-1,a-j,b-j} D_q^j[fg].								\nonumber
\end{eqnarray} 
\end{prop}

\begin{proof}
The proof of this follows the same arguments as in \cite{CEM}. By exchanging order of summations in each of the terms in (\ref{DD-Euler}), we can combine the first and second terms, and the third and fourth terms. Then some basic binomial identities are used.   
\end{proof}

However, a natural question to ask is, whether we can resolve products, like for instance $D_q[f]P_q[g]$. And indeed, from (\ref{qLeibniz}) we conclude quickly that:
\begin{equation}\label{eq:pqdq}
	D_q[f]P_q[g] = D_q\big[ f P_q[g] \big] + D_q[f]g - fg.
\end{equation}
Equations \eqref{qRBR}, \eqref{qLeibniz} and \eqref{eq:pqdq} are indeed equivalent. A less obvious exercise is the product:
\begin{eqnarray*}
	D_qD_q[f]P_qP_q[g] &=& D_q\big[ D_q[f]P_qP_q[g]\big] + D_qD_q[f]P_q[g] - D_q[f]P_q[g]\\
					 &=&  D_q\Big[ D_q\big[ f P_qP_q[g] \big] +  D_q[f]P_q[g] - fP_q[g]\big]\Big]\\
					 & & + D_q\big[ D_q[f]P_q[g]\big] + D_qD_q[f]g - 2D_q[f]g - D_q\big[ fP_q[g]\big] +fg\\
					 &=& D_qD_q\big[ f P_qP_q[g] \big] + 2 D_q\big[ D_q\big[ f P_q[g] \big] + D_q[f]g - fg\big] \\
					 & & -2 D_q\big[ fP_q[g]\big] - 2D_q[f]g + D_qD_q[f]g + fg   \\
					 &=& D_qD_q\big[ f P_qP_q[g] \big] + 2 D_qD_q\big[ f P_q[g] \big] + 2 D_q\big[ D_q[f]g\big]\\
					 & & - 2D_q[fg] -2 D_q\big[ fP_q[g]\big] - 2D_q[f]g + D_qD_q[f]g + fg.     	
\end{eqnarray*}

\noindent
The recursion for a general product of this form is given by:
$$
	D_q^a[f]P_q^b[g] = D_q\big[ D_q^{a-1}[f]P_q^b[g]\big] + D_q^{a}[f]P_q^{b-1}[g] - D_q^{a-1}[f]P_q^{b-1}[g]\big].
$$

The closed expression is given in the next proposition.

\begin{prop}
\label{prop:EulerDP}
Let $1<a\leq b \in \mathbb{N}$.  
\begin{eqnarray*} 
\label{DP-Euler-slim}
	D_q^a[f]P_q^b[g] &=& 
	\sum_{j=0}^{a}\ \sum_{i=1}^{b-a+j}  
		(-1)^{a-j} {b-1-i+j \choose a-1}{a \choose j} D_q^j\big[fP^i_q[g]\big]	 					\\
	& & \hspace{0.5cm}
	 	+ \sum_{k=1}^{a}\ \sum_{i=1}^{k}  
		(-1)^{a-k} {b-1-i+k \choose b-1}{b \choose a-k} D_q^{k-i}\big[D^i_q[f]g\big]	 				\nonumber\\
	& & \hspace{1cm}
 	 + \sum_{j=0}^{a-1} (-1)^{a-j}  {b-1+j \choose j,a-1-j,b-a+j} D_q^j[fg].	 						\nonumber
\end{eqnarray*} 	
\end{prop}

\begin{proof}
Again this follows the same argument given in the proof of Proposition (\ref{DD-Euler}). See \cite{DP,CEM}. 
\end{proof}

Further below we will see that these identities provide $q$-generalizations of Euler's decomposition formula for the modified $q$-analog (\ref{mod-qMZV-z}) at values $n_1,n_2, \ldots, n_k \in \mathbb{Z}$.


\section{\lowercase{$q$}-analogs of Multiple Zeta Values}
\label{sect:qMZV}

Recall that for positive natural numbers $n_1, \ldots, n_k \in \mathbb{N}$, $n_1>1$, classical multiple zeta values (MZVs) of depth $k$ and weight $w:=n_1+n_2+\cdots+n_k$ are defined as $k$-fold iterated infinite series \cite{Hoffman2,Waldschmidt,Zagier,Zudilin}:
\begin{eqnarray}
	\zeta(n_1,\dots,n_k)&:=&\sum_{m_1  > \dots > m_k > 0}
						\frac{1}{ m_1^{n_1}\cdots m_k^{n_k}} \label{MZVs}\\
				   &=&\idotsint\limits_{0\leq t_w \leq \cdots \leq t_1 \leq 1} 
				   		\frac{dt_1}{\tau_1(t_1)} \cdots \frac{dt_w}{\tau_w(t_w)}, \label{ChenRep}
\end{eqnarray}
where $\tau_i(u) = 1-u$ if $i \in \{h_1,h_2,\ldots,h_k\}$, $h_j:=n_1+n_2+\cdots+n_j$, and $\tau_i(u) = u$ otherwise.


\subsection{Iterated Jackson integrals and $q$-multiple zeta values}
\label{ssect:JacksonMZV}

Following \cite{CEM} we define the functions $x:=1/\id$, $y:=1/(1-\id)$, and $\bar{y}:=\id/(1-\id)$, such that:
\begin{equation*}
	x(t)=\frac{1}{t}, \quad y(t)=\frac{1}{1-t}, \quad \bar{y}(t)=\frac{t}{1-t}.
\end{equation*}
Recall the common notation for $q$-numbers, $[m]_q:=\frac{1-q^m}{1-q}$.\\

Replacing the Riemann integrals in (\ref{ChenRep}) by Jackson integrals (\ref{qInt}), we arrive at a $q$-analog of MZVs, which was considered by Y.~Ohno, J.~Okuda and W.~Zudilin in \cite{OhOkZu}. It is defined for positive natural numbers $n_i \in \mathbb{N}$, $n_1>1$, $w:=n_1+\cdots+n_k$, in terms of iterated Jackson integrals evaluated at $q$:  
$$
	\mathfrak{z}_q(n_1,\ldots,n_k)	
     				:= J\Big[ \rho_1 J\big[ \rho_2 \cdots J[\rho_w] \cdots \big] \Big](q), 
$$			
where $\rho_i(t) = y(t)$ if $i \in \{h_1,h_2,\ldots,h_k\}$, $h_j:=n_1+n_2+\cdots+n_j$, and $\rho_i(t) = x(t)$ otherwise. Writing this out in detail yields:			
\begin{eqnarray}
     \mathfrak{z}_q(n_1,\ldots,n_k)	
     			&=&(1-q)^w\underbrace{P_q \: \big[P_q \: [\cdots P_q}_{n_1}\: [\bar{y}
                                     \cdots
                                     \underbrace{P_q \: [P_q\:[ \cdots P_q}_{n_k}\: [\bar{y}]]]]
                                     \cdots ]\big](q)										\nonumber\\
                       	&=&	(1-q)^w \sum_{m_1  > \dots > m_k > 0}
                             	\frac{q^{m_1}}{(1-q^{m_1})^{n_1}\cdots (1-q^{m_k})^{n_k}}		\nonumber\\
         		&=& \sum_{m_1 > \dots > m_k > 0}
				\frac{q^{m_1 }}{[m_1]_q^{n_1}\cdots [m_k]_q^{n_k}}.  			\label{2qMZVs}
\end{eqnarray}
In the introduction we mentioned the modified $q$MZV:
\begin{equation}
\label{modz}
    	\bar{\mathfrak{z}}_q(n_1,\ldots,n_k)
         :=\sum_{m_1  > \dots > m_k> 0}
			\frac{q^{m_1}}{(1-q^{m_1})^{n_1}\cdots (1-q^{m_k})^{n_k}},	
\end{equation}
for which we have:
\begin{equation}\label{weighting}
(1-q)^w \bar{\mathfrak{z}}_q(n_1,\ldots,n_k) = \mathfrak{z}_q(n_1,\ldots,n_k).
\end{equation}
In the following we will mainly work with these modified $q$MZVs. From (\ref{2qMZVs}) it becomes clear that:
$$
	\bar{\mathfrak{z}}_q(n_1,\ldots,n_k) = 
				P_q^{n_1}\big[\bar{y} \cdots P_q^{n_k}\big[\bar{y}]\cdots \big](q),
$$
where we use the $n$-fold composition:
$$
	P_q^{n} := P_q \circ \cdots \circ P_q.
$$


\subsection{Convergence issues and extension to integer arguments of any sign}
\label{ssect:neg-qMZVs-z}

Observe that:
$$
	\bar{\mathfrak{z}}_q(0)=P_q^{0}[\bar{y}](q)=\sum_{m>0}q^m=\bar{y}(q)=\frac{q}{1-q},
$$
which makes perfect sense as a formal series in $q$, the specialization of which is well-defined for any complex $q$ with $|q|<1$. More generally \eqref{2qMZVs} and \eqref{modz} make sense as an element of $q\mathbb Q[[q]]$, i.e. as a formal series in $q$ without constant term, for any $n_1,\ldots, n_k\in\mathbb Z$, and for a complex number $q$ with $|q|<1$ the series converges. This follows from:
\begin{eqnarray*}
	|\bar{\mathfrak{z}}_q(n_1,\ldots,n_k)|	
				&\le&	\sum_{m_1  > \dots > m_k> 0}
					\frac{|q|^{m_1}}{(1-|q|^{m_1})^{n_1}\cdots (1-|q|^{m_k})^{n_k}}\\
				&\le &	\sum_{m_1  > \dots > m_k> 0}
					\frac{|q|^{m_1}}{(1-|q|)^{\tilde{n}_1}\cdots (1-|q|)^{\tilde{n}_k}}\\
				&\le&	(1-|q|)^{-\tilde{w}}\sum_{m'_1,\ldots ,m'_k>0}|q|^{m'_1+\cdots+m'_k}
\end{eqnarray*}
with $\tilde{n}_i:=\mop{sup}(0,n_i)=\frac{1}{2}(n_i + |n_i|)$ for $i=1,\ldots ,k$, such that $\tilde{w}= \sum_{i=1}^k \tilde{n}_i $, and hence:
\begin{equation*}
	|\bar{\mathfrak{z}}_q(n_1,\ldots,n_k)|\le |q|^k(1-|q|)^{-\tilde{w}-k}.
\end{equation*}
Also, for non-modified $q$MZVs:
\begin{equation}
|{\mathfrak{z}}_q(n_1,\ldots,n_k)|\le |q|^k(1-|q|)^{-\tilde{w}-k}|1-q|^{w}.
\end{equation}
\begin{prop}
For $n_1\ge 2$ and $n_2,\ldots,n_k\ge 1$, in the limit $q \to 1$ when $|q|<1$ and $\mop{Arg}(1-q)\in[-\frac {\pi}{2}+\varepsilon,\,\frac{\pi}{2}-\varepsilon]$ for some $\varepsilon>0$, the $k$-fold iterated sum \eqref{2qMZVs} converges to the corresponding classical MZV of depth $k$ and weight $w$.
\end{prop} 
\begin{proof}
This is an immediate application of Abel's limit theorem for power series of convergence radius $1$ \cite[p. 41-42]{Ahlfors}.
\end{proof}
The $q$-parameter may then be considered a regularization of MZVs for arguments in $\mathbb Z$. Using $D_q=P_q^{-1}$ we immediately see that for $n_i \in \mathbb{Z}$ we have:
\begin{equation}
	\bar{\mathfrak{z}}_q(n_1,\ldots,n_k) = 
				P_q^{n_1}\big[\bar{y} \cdots P_q^{n_k}\big[\bar{y}]\cdots \big](q).
\end{equation}
For example:
$$
	\bar{\mathfrak{z}}_q(0,0)=\sum_{m_1>m_2>0}q^{m_1}=\sum_{m>0}(m-1)q^m=\left(\frac{q}{1-q}\right)^2, 
	\qquad {\rm{and}} \qquad
	\bar{\mathfrak{z}}_q(\underbrace{0,\ldots,0}_{k\smop{ times}})=\left(\frac{q}{1-q}\right)^k.
$$	
For $a < 0$:	
\begin{eqnarray*}
	\bar{\mathfrak{z}}_q(a)= D_q^{|a|}(\bar{y})(q)=\sum_{m>0}q^{m}(1-q^{m})^{|a|},
\end{eqnarray*}
and for $a>0$:
$$
	\bar{\mathfrak{z}}_q(a,0)	=\sum_{m_1>m_2>0}\frac{q^{m_1}}{(1-q^{m_1})^a}
						=\sum_{m>0}\frac{(m-1)q^{m}}{(1-q^{m})^a}.
$$
\noindent Finally, it will be also convenient to express our $q$-multiple zeta values in terms of the parameter $q^{-1}$ whenever possible:
\begin{prop}
The $q$-multiple zeta values ${\mathfrak{z}}_q(n_1,\ldots,n_k)$ and $\overline{{\mathfrak{z}}}_q(n_1,\ldots,n_k)$ make sense as a series in $\mathbb Q[[q^{-1}]]$ for any $(n_1,\ldots,n_k)$ with $n_j\ge 1$ and $n_1\ge 2$.
\end{prop}
  \begin{proof}
  This comes from the straightforward computation:
  \begin{equation}
  {\mathfrak{z}}_q(n_1,\ldots,n_k)=\sum_{m_1>\cdots>m_k>0}\frac{(q^{-1})^{-m_1}(q^{-1})^{(m_1-1)n_1+\cdots +(m_k-1)n_k}}{[m_1]_{q^{-1}}^{n_1}\cdots[m_k]_{q^{-1}}^{n_k}}.
  \end{equation}
  \end{proof}

\subsection{The $q$-shuffle structure}
\label{ssect:q-shuffle}

Let $W$ be the set of words on the alphabet $\wt X:=\{d,y,p\}$ ending with $y$, subject to $dp=pd=\un$ (where $\un$ stands for the empty word). We shall also use the notation $p^{-1}=d$. Any non-empty word $v$ in $W$ writes in a unique way:
\begin{equation}
	v=p^{n_1}y\cdots p^{n_k}y,
\end{equation}
with $k> 0$ and $n_1,\ldots, n_k\in\mathbb Z$. The length of the word $v$ is given by:
\begin{equation}
	\ell(v)=k+|n_1|+\cdots +|n_k|.
\end{equation}
For later use we introduce the notations:
\begin{eqnarray}
\bar{\mathfrak{z}}_q^{\sshu}(p^{n_1}y\cdots p^{n_k}y)&:=&\bar{\mathfrak{z}}_q(n_1,\ldots,n_k),\\
{\mathfrak{z}}_q^{\sshu}(p^{n_1}y\cdots p^{n_k}y)&:=&{\mathfrak{z}}_q(n_1,\ldots,n_k).
\end{eqnarray}
The \textsl{$q$-shuffle product} is given on $\mathbb Q.W$ recursively (with respect to the length of words) by $\un\shu v=v\shu\un=v$ for any word $v$, and:
\begin{eqnarray}
	(yv)\shu u=v\shu (yu)&=&y(v\shu u),						\label{shuffle-y}\\
			pv\shu pu&=&p(v\shu pu)+p(pv\shu u) - p(v\shu u),	\label{shuffle-p}\\
			dv\shu du&=&v\shu du + dv\shu u - d(v\shu u),		\label{shuffle-d}\\
	dv\shu pu=pu\shu dv&=&d(v\shu pu)+dv\shu u - v\shu u			\label{shuffle-dp}
\end{eqnarray}
for any words $v$ and $u$ in $W$. Equations come from an abstraction of Equations \eqref{qRBR}, \eqref{qLeibniz} and \eqref{eq:pqdq} respectively.

\begin{thm}\label{q-shuffle-structure}
The $q$-shuffle product is commutative and associative. Moreover for any $v,u \in W$ we have:
\begin{equation*}
	\bar{\mathfrak z}_q^{\sshu}(v)\bar{\mathfrak z}_q^{\sshu}(u)
	=\bar{\mathfrak z}_q^{\sshu}(v\shu u).
\end{equation*}
\end{thm}

\begin{proof}
Proving commutativity is done by showing $u\! \shu v = v\! \shu u$ using induction on the sum $\ell(u)+\ell(v)$ of the lengths of the two words $u$ and $v$. It is left to the reader. We now prove the associativity relation $(u\shu v)\shu z = u \shu (v\shu z)$ by induction on the sum $\ell(u)+\ell(v)+\ell(z)$ of the lengths of the three words $u$, $v$ and $z$.\\

If one of the words is empty there is nothing to prove. Otherwise we write $u=\alpha a$, $v=\beta b$, and $z=\gamma c$ where $\alpha$, $\beta$ or $\gamma$ can be the letters $p$, $d$ or $y$. This yields theoretically 27 different cases, which however will reduce substantially:

\begin{itemize}

\item\textbf{First case: one of the letters is a $y$}. Using \eqref{shuffle-y} repeatedly as well as the induction hypothesis we have:
\begin{eqnarray*}
	(ya\shu v)\shu z	&=&\big(y(a\shu v)\big)\shu z\\
				&=&y\big((a\shu v)\shu z\big)\\
				&=&y\big(a\shu(v\shu z)\big)\\
				&=&ya\shu(v\shu z).
\end{eqnarray*}
The similar cases amount to this one by using commutativity.\\

\item\textbf{Second case: $\alpha=\beta=\gamma=d$}. We use \eqref{shuffle-d} repeatedly as well as commutativity, and we freely omit parentheses when using the induction hypothesis. On one hand we have:
\allowdisplaybreaks{
\begin{eqnarray*}
	(da\shu db)\shu dc	&=&\big(a\shu db+da\shu b-d(a\shu b)\big)\shu dc\\
				&=&a\shu db\shu dc+b\shu da\shu dc-d(a\shu b)\shu dc\\
				&=&a\shu\big(b\shu dc+db\shu c-d(b\shu c)\big)
						+b\shu\big(a\shu dc+da\shu c-d(a\shu c)\big)\\
				&& -a\shu b\shu dc-d(a\shu b)\shu c+d(a\shu b\shu c)\\
				&=&\underbrace{a\shu db\shu c}_1-\underbrace{a\shu d(b\shu c)}_2
					+\underbrace{b\shu a\shu dc}_3
						+\underbrace{b\shu da\shu c}_4
							-\underbrace{b\shu d(a\shu c)}_5\\
				&&-\underbrace{d(a\shu b)\shu c}_6+\underbrace{d(a\shu b\shu c)}_7.
\end{eqnarray*}}
On the other hand,
\allowdisplaybreaks{
\begin{eqnarray*}
	da\shu(db\shu dc)	&=&da\shu\big(b\shu dc+db\shu c-d(b\shu c)\big)\\
				&=&da\shu dc\shu b+da\shu db\shu c-da\shu d(b\shu c)\\
				&=&a\shu dc\shu b+da\shu c\shu b-d(a\shu c)\shu b
					+a\shu db\shu c+da\shu b\shu c\\
				&& -d(a\shu b)\shu c-a\shu d(b\shu c)-da\shu b\shu c +d(a\shu b\shu c).\\
				&=&\underbrace{a\shu dc\shu b}_3-\underbrace{d(a\shu c)\shu b}_5
					+\underbrace{a\shu db\shu c}_1+\underbrace{da\shu b\shu c}_4
						-\underbrace{d(a\shu b)\shu c}_6\\
				&&-\underbrace{a\shu d(b\shu c)}_2 +\underbrace{d(a\shu b\shu c)}_7.
\end{eqnarray*}}
Hence both expressions coincide.\\

\item\textbf{Third case: one $p$ and two $d$'s}. We can suppose that $\alpha=\beta=d$ and $\gamma=p$: the other cases will follow by commutativity. We use both \eqref{shuffle-d} and \eqref{shuffle-dp}. On one hand we have:
\allowdisplaybreaks{
\begin{eqnarray*}
	(da\shu db)\shu pc	&=&\big(a\shu db+da\shu b-d(a\shu b)\big)\shu pc\\
				&=&a\shu db\shu pc+b\shu da\shu pc-d(a\shu b)\shu pc\\
				&=&a\shu d(b\shu pc)+a\shu db\shu c-a\shu b\shu c
					+b\shu d(a\shu pc)+b\shu da\shu c\\
				&&-b\shu a\shu c-d(a\shu b\shu pc)-d(a\shu b)\shu c+a\shu b\shu c\\
				&=&\underbrace{a\shu d(b\shu pc)}_1+\underbrace{a\shu db\shu c}_2
					-\underbrace{a\shu b\shu c}_3+\underbrace{b\shu d(a\shu pc)}_4
						+\underbrace{b\shu da\shu c}_5\\
				&&-\underbrace{d(a\shu b\shu pc)}_6-\underbrace{d(a\shu b)\shu c}_7.
\end{eqnarray*}}
On the other hand,
\allowdisplaybreaks{
\begin{eqnarray*}
	da\shu(db\shu pc)	&=&da\shu\big(d(b\shu pc)+db\shu c-b\shu c)\\
				&=&da\shu d(b\shu pc)+da\shu db\shu c - da\shu b\shu c\\
				&=&a\shu d(b\shu pc)+da\shu pc\shu b-d(a\shu b\shu pc)
					+a\shu db\shu c\\
				&&+da\shu b\shu c-d(a\shu b)\shu c-da\shu b\shu c\\
				&=&\underbrace{a\shu d(b\shu pc)}_1+\underbrace{d(a\shu pc)\shu b}_4
					+\underbrace{da\shu c\shu b}_5-\underbrace{a\shu c\shu b}_3
						\\
				&&-\underbrace{d(a\shu b\shu pc)}_6+\underbrace{a\shu db\shu c}_2-\underbrace{d(a\shu b)\shu c}_7.
\end{eqnarray*}}

\item\textbf{Fourth case: two $p$'s and one $d$}. Again we can suppose that $\alpha=\beta=p$ and $\gamma=d$: the other cases will follow by commutativity. We use both \eqref{shuffle-p} and \eqref{shuffle-dp}. One one hand we have:
\allowdisplaybreaks{
\begin{eqnarray*}
	(pa\shu pb)\shu dc	&=&p(a\shu pb)\shu dc+p(pa\shu b)\shu dc-p(a\shu b)\shu dc\\
				&=&d\big(c\shu p(a\shu pb)\big)+dc\shu a\shu pb-a\shu pb\shu c
					+d\big(c\shu p(b\shu pa)\big)\\
				&& +dc\shu b\shu pa-b\shu pa\shu c-d\big((c\shu p(a\shu b)\big)-dc\shu a\shu b+c\shu a\shu b.\\
				&=&d\big(c\shu p(a\shu pb)\big)+a\shu d(c\shu pb)+a\shu dc\shu b-a\shu c\shu b
					-a\shu pb\shu c\\
				&& +d\big(c\shu p(b\shu pa)\big)+b\shu d(c\shu pa)+b\shu dc\shu a
					-b\shu c\shu a-b\shu pa\shu c\\
				&& -d\big((c\shu p(a\shu b)\big)-dc\shu a\shu b+c\shu a\shu b.\\
				&=&\underbrace{d\big(c\shu p(a\shu pb)\big)}_1+\underbrace{a\shu d(c\shu pb)}_2
					-\underbrace{a\shu c\shu b}_3-\underbrace{a\shu pb\shu c}_4
						+\underbrace{d\big(c\shu p(b\shu pa)\big)}_5\\
				&& +\underbrace{b\shu d(c\shu pa)}_6+\underbrace{b\shu dc\shu a}_7
					-\underbrace{b\shu pa\shu c}_8-\underbrace{d\big(c\shu p(a\shu b)\big)}_9. 
\end{eqnarray*}}
On the other hand:
\allowdisplaybreaks{
\begin{eqnarray*}
	pa\shu(pb\shu dc)	&=&pa\shu \big(d(c\shu pb)+dc\shu b-c\shu b\big)\\
				&=&d(c\shu pb\shu pa)+d(c\shu pb)\shu a-c\shu pb\shu a\\
				&&+d(c\shu pa)\shu b+dc\shu a \shu b-a\shu c\shu b-pa\shu c\shu b\\
				&=&\underbrace{d\big((c\shu p(b\shu pa)\big)}_5
					+\underbrace{d\big(c\shu p(pb\shu a)\big)}_1
					-\underbrace{d\big(c\shu p(b\shu a)\big)}_9
					+\underbrace{d(c\shu pb)\shu a}_2 \\
				&&-\underbrace{c\shu pb\shu a}_4+\underbrace{d(c\shu pa)\shu b}_6+\underbrace{dc\shu a \shu b}_7
					-\underbrace{a\shu c\shu b}_3-\underbrace{pa\shu c\shu b}_8.
\end{eqnarray*}}

\item\textbf{Fifth case: $\alpha=\beta=\gamma=p$}. This is the ordinary quasi-shuffle case. We detail it for completeness, using \eqref{shuffle-p}. One one hand we have:
\allowdisplaybreaks{
\begin{eqnarray*}
	(pa\shu pb)\shu pc	&=&p(pa\shu b)\shu pc+p(pb\shu a)\shu pc-p(a\shu b)\shu pc\\
				&=&p\Big(pa\shu b\shu pc+p(pa\shu b)\shu v-pa\shu b\shu c
					+pb\shu a\shu pc+p(pb\shu a)\shu c\\
				&&-pb\shu a\shu c-p(a\shu b)\shu c-a\shu b\shu pc+a\shu b\shu c\Big)\\
				&=&p\Big(\underbrace{b\shu p(a\shu pc)}_1+\underbrace{b\shu p(pa \shu c)}_2
					-\underbrace{b\shu p(a\shu c)}_3+\underbrace{p(pa\shu b)\shu c}_4
						-\underbrace{pa\shu b\shu c}_5\\
				&&+\underbrace{a\shu p(b\shu pc)}_6+\underbrace{a\shu p(pb \shu c)}_7
					-\underbrace{a\shu p(b\shu c)}_8+\underbrace{p(pb\shu a)\shu c}_9
						-\underbrace{pb\shu a\shu c}_{10}\\
				&&-\underbrace{p(a\shu b)\shu c}_{11}-\underbrace{a\shu b\shu pc}_{12}
					+\underbrace{a\shu b\shu c}_{13}\Big).
\end{eqnarray*}}
On the other hand:
\allowdisplaybreaks{
\begin{eqnarray*}
	pa\shu(pb\shu pc)	&=&(pb\shu pc)\shu pa\\
				&=& p\Big(\underbrace{c\shu p(b\shu pa)}_4+\underbrace{c\shu p(pb \shu a)}_9
					-\underbrace{c\shu p(b\shu a)}_{11}+\underbrace{p(pb\shu c)\shu a}_7
					-\underbrace{pb\shu c\shu a}_{10}\\
				&&+\underbrace{b\shu p(c\shu pa)}_2+\underbrace{b\shu p(pc \shu a)}_1
					-\underbrace{b\shu p(c\shu a)}_3+\underbrace{p(pc\shu b)\shu a}_6
						-\underbrace{pc\shu b\shu a}_{12}\\
				&&-\underbrace{p(b\shu c)\shu a}_8-\underbrace{b\shu c\shu pa}_5
					+\underbrace{b\shu c\shu a}_{13}\Big).
\end{eqnarray*}}
\end{itemize}
This concludes the proof of the associativity.\\

Now let us call a differential Rota--Baxter algebra \textsl{invertible} if the Rota--Baxter operator $P$ and the differential $D$ are mutually inverse. Then $(\mathbb Q.W,\shu)$ is the free invertible differential Rota--Baxter algebra of weight $-1$ with one generator. Indeed, the generator is $y$, the Rota--Baxter  operator (respectively the differential) is left concatenation by the letter $p$ (respectively $d$), and identities \eqref{shuffle-y}, \eqref{shuffle-d}, \eqref{shuffle-p}, \eqref{shuffle-dp} guarantee the weight $-1$ differential Rota--Baxter identities. This object is very different from (and much smaller than) the free differential Rota--Baxter algebra with one generator constructed in \cite{GK}. The map 
\begin{eqnarray*}
	\Cal Z:\mathbb Q.W		&\longrightarrow& \Cal A\\
	p^{n_1}y\cdots p^{n_k}y 	&\longmapsto & 
					P_q^{n_1}\big[\bar yP_q^{n_2}[\bar y\cdots P_q^{n_k}[\bar y]\cdots]\big]
\end{eqnarray*}
is the unique map of invertible differential Rota--Baxter algebras of weight $-1$ such that $\Cal Z(y)=\bar y$ (recall $\bar y(t,q)=y(t):=\frac t{1-t}$). The second assertion of Theorem \ref{q-shuffle-structure} immediately comes from the fact that for any word $v$:
\begin{equation*}
	\bar {\mathfrak z}_q^{\sshu}(v)=\Cal Z(v)(t,q)\restr{t=q}.
\end{equation*}
\end{proof}

\begin{rmk}
{\rm Considering what happens with ordinary shuffle or quasi-shuffle products, it would be nice to have a purely combinatorial interpretation of the product $\shu$.}
\end{rmk}


\subsection{Euler decomposition formulas}
\label{ssect:qEulerDecomp}

Recall the identities (\ref{DD-Euler-slim}) and (\ref{DP-Euler-slim}) in Propositions \ref{prop:EulerDD} and \ref{prop:EulerDP}, respectively, from which we derive $q$-generalization of Euler's decomposition formulas for $q$MZVs, which complement (\ref{Z-Euler}). For $1<a \leq b$:
\allowdisplaybreaks{
\begin{eqnarray} 
\label{DD-Euler-slim}
	\bar{\mathfrak{z}}_q(-a)\bar{\mathfrak{z}}_q(-b)&=& 
	\sum_{j=0}^{a}\ \sum_{i=1}^{b-j}  
		(-1)^j {a+b-1-i-j \choose a-1}{a \choose j} \bar{\mathfrak{z}}_q(-j,-i)	 					\\
	& & \hspace{0.5cm}
	 	+ \sum_{j=0}^{b}\ \sum_{i=1}^{\max(1,a-j)}  
		(-1)^j {a+b-1-i-j \choose b-1}{b \choose j} \bar{\mathfrak{z}}_q(-j,-i)	 				\nonumber\\
	& & \hspace{1cm}
		+ \sum_{j=1}^{a} (-1)^j  {a+b-1-j \choose j-1,a-j,b-j} \bar{\mathfrak{z}}_q(-j,0),								\nonumber
\end{eqnarray}} 
and
\allowdisplaybreaks{
\begin{eqnarray} 
\label{DP-Euler-slim}
	\bar{\mathfrak{z}}_q(-a)\bar{\mathfrak{z}}_q(b) &=& 
	\sum_{j=0}^{a}\ \sum_{i=1}^{b-a+j}  
		(-1)^{a-j} {b-1-i+j \choose a-1}{a \choose j} \bar{\mathfrak{z}}_q(-j,i)	 					\\
	& & \hspace{0.5cm}
	 	+ \sum_{k=1}^{a}\ \sum_{i=1}^{k}  
		(-1)^{a-k} {b-1-i+k \choose b-1}{b \choose a-k} \bar{\mathfrak{z}}_q(-k+1,-i)	 				\nonumber\\
	& & \hspace{1cm}
		+ \sum_{j=0}^{\min(a-1,b-a+1)} (-1)^{a-j}  {b-1+j \choose j,a-1-j,b-a-j} 
			\bar{\mathfrak{z}}_q(-j,0).		\nonumber
\end{eqnarray}} 	

In \cite{CEM} we have seen how the $\delta:=q\frac{d}{dq}$ derivation terms entered systematically the $q$-generalization (\ref{Z-Euler}) of Euler's decomposition formula. Further below, in subsection \ref{ssect:diffalg}, we will show how for $j>0$, one can rewrite the third summands on the right hand side of each of the above identities in terms of linear combinations of modified $q$MZVs and derivation terms $\delta\mathfrak{z}_q$.


\subsection{The $q$-quasi-shuffle structure}
\label{ssect:qqshuf}

In \cite{CEM} we saw that the product of two such modified $q$MZVs of weight $a,b>1$, using the sum representation (\ref{2qMZVs}), satisfies the following identity:
\begin{equation}
\label{qshz}
    	\bar{\mathfrak{z}}_q(a)\bar{\mathfrak{z}}_q(b) = \bar{\mathfrak{z}}_q(a,b) +\bar{\mathfrak{z}}_q(b,a) +\bar{\mathfrak{z}}_q(a+b)
				-\bar{\mathfrak{z}}_q(a,b-1) - \bar{\mathfrak{z}}_q(b,a-1) - \bar{\mathfrak{z}}_q(a+b-1).
\end{equation}
The same formula holds for any $a,b\in\mathbb Z$, as the following computation shows:
\allowdisplaybreaks{
\begin{eqnarray*}
	\lefteqn{\bar{\mathfrak{z}}_q(a)\bar{\mathfrak{z}}_q(b)}\\
	&=& 
	\sum_{l_1>l_2>0} q^{l_1+l_2} (1-q^{l_1})^{-a}(1-q^{l_2})^{-b} +
	\sum_{l_1>l_2>0} q^{l_1+l_2} (1-q^{l_1})^{-b}(1-q^{l_2})^{-a} 
	+ \sum_{l>0}  q^{2l}(1-q^{l})^{-a-b}	\\
	 &=&
	\sum_{l_1>l_2>0} \big(q^{l_1} - q^{l_1} (1-q^{l_2})\big)(1-q^{l_1})^{-a}(1-q^{l_2})^{-b}\\
	&&+\sum_{l_1>l_2>0} \big(q^{l_1} - q^{l_1} (1-q^{l_2})\big)(1-q^{l_1})^{-b}(1-q^{l_2})^{-a}		  
		- \sum_{l>0}  \big(q^{l} - q^{l} (1-q^{l})\big)(1-q^{l})^{-a-b}					\\
	&=& \bar{\mathfrak{z}}_q(a,b) + \bar{\mathfrak{z}}_q(b,a) + \bar{\mathfrak{z}}_q(a+b)  
		- \bar{\mathfrak{z}}_q(a,b-1) - \bar{\mathfrak{z}}_q(b,a-1) - \bar{\mathfrak{z}}_q(a+b-1).						\label{qshz2a}
\end{eqnarray*}}
We can formulate the above $q$-quasi-shuffles in terms of a quasi-shuffle like algebra. Let $\wt Y$ be the alphabet $\{z_n,\,n\in\mathbb Z\}$. We denote by $\wt Y^*$ the set of words with letters in $\wt Y$, and by $\mathbb Q\langle \tilde{Y}\rangle$ the free associative algebra on $\wt Y$, which is freely generated as a $\mathbb Q$-vector space by $\wt Y^*$. We equip $\wt Y$ with the internal commutative associative product $[z_iz_j]:=z_{i+j}$. For later use we introduce the notations:
\begin{eqnarray}
\bar{\mathfrak{z}}_q^{\sqshu}(z_{n_1}\cdots z_{n_k})&:=&\bar{\mathfrak{z}}_q(n_1,\ldots,n_k),\\
{\mathfrak{z}}_q^{\sqshu}(z_{n_1}\cdots z_{n_k})&:=&{\mathfrak{z}}_q(n_1,\ldots,n_k).
\end{eqnarray}
On $\mathbb{Q}\langle \tilde{Y}\rangle$ we consider the ordinary quasi-shuffle product $*$, recursively defined by:
$$
	av * bv' := a(v * bv') + b(av*v') + [ab](v*v'). 
$$
This product is known to be commutative and associative \cite{Hoffman-qsh}. Now we consider the linear operator $T$ on $\mathbb Q\langle \tilde{Y}\rangle$ defined by:
\begin{equation*}
	T(z_nv):=z_nv - z_{n-1}v.
\end{equation*}
It is obviously injective. For any $m,n\in\mathbb Z$ and for any $u,v\in\wt Y^*$ we compute:
\begin{eqnarray*}
	T(z_mu)*T(z_nv)	&=&(z_m-z_{m-1})u*(z_n-z_{n-1})v\\
					&=&(z_m-z_{m-1})\big(u*(z_n-z_{n-1})v\big)
					+(z_n-z_{n-1})\big(v*(z_m-z_{m-1})u\big)\\
					&&\hbox {\hskip 10mm}+\big((z_{m+n}-z_{m+n-1})
						-(z_{m+n-1}-z_{m+n-2})\big)(u*v)\\
					&=&T\Big(z_m\big(u*T(z_nv)\big)+z_n\big(T(z_mu)*v)\big)+T\big(z_{m+n}(u*v)\big)\Big).
\end{eqnarray*}
We then define our $q$-quasi-shuffle product by $T(u\qshu v)=Tu*Tv$ for any words $u,v$. In view of the computation above it writes:
\begin{equation*}
	z_mu \qshu z_nv=z_m\big(u*T(z_nv)\big)+z_n\big(T(z_mu)*v\big)+(z_{m+n}-z_{m+n-1})(u*v).
\end{equation*} 
In particular we have:
\begin{eqnarray*}
	z_m\qshu z_n	&=&z_m(Tz_n)+z_n(Tz_m)+Tz_{m+n}\\
				&=&z_mz_n+z_nz_m+z_{m+n}-z_mz_{n-1}-z_nz_{m-1}-z_{m+n-1}.
\end{eqnarray*}
\begin{prop}
The $q$-quasi-shuffle product {\rm $\qshu$} is commutative and associative. Moreover for any $u,v\in \wt Y$ we have:
{\rm
\begin{equation*}
	\bar{\mathfrak z}_q^{\sqshu}(u)\bar{\mathfrak z}_q^{\sqshu}(v)
		=\bar{\mathfrak z}_q^{\sqshu}(u\qshu v).
\end{equation*}
}
\end{prop}

\begin{proof}
Commutativity and associativity of $\qshu$ come from the injectivity of $T$ and from the fact that the ordinary quasi-shufle product $*$ is commutative and associative. The second assertion has been already proven when $u$ and $v$ are two letters. The computation for two words is entirely similar:
\allowdisplaybreaks{
\begin{eqnarray*}
	\lefteqn{\bar{\mathfrak z}_q^{\sqshu}(z_{n_1}\cdots z_{n_r})
	\bar{\mathfrak z}_q^{\sqshu}(z_{n_{r+1}}\cdots z_{n_{r+s}})}\\
	&=&\sum_{{k_1>\cdots >k_r\atop k_{r+1}>\cdots>k_{r+s}}}
		\frac{q^{k_1+k_{r+1}}}{(1-q^{k_1})^{n_1}\cdots(1-q^{k_{r+s}})^{n_{r+s}}}\\
	&=&\sum_{{k_1>k_{r+1},\,k_1>\cdots >k_r\atop k_{r+1}>\cdots>k_{r+s}}}
		\frac{q^{k_1+k_{r+1}}}{(1-q^{k_1})^{n_1}\cdots(1-q^{k_{r+s}})^{n_{r+s}}}
		+\sum_{{k_1<k_{r+1},\,k_1>\cdots >k_r\atop k_{r+1}>\cdots>k_{r+s}}}
		\frac{q^{k_1+k_{r+1}}}{(1-q^{k_1})^{n_1}\cdots(1-q^{k_{r+s}})^{n_{r+s}}}\\
	&&\hskip 10mm+\sum_{{k_1=k_{r+1},\,k_1>\cdots >k_r\atop k_{r+1}>\cdots>k_{r+s}}}
		\frac{q^{2k_1}}{(1-q^{k_1})^{n_1}\cdots(1-q^{k_{r+s}})^{n_{r+s}}}\\
	&=&\sum_{{k_1>k_{r+1},\,k_1>\cdots >k_r\atop k_{r+1}>\cdots>k_{r+s}}}
		\frac{q^{k_1}-q^{k_1}(1-q^{k_{r+1})}}{(1-q^{k_1})^{n_1}\cdots(1-q^{k_{r+s}})^{n_{r+s}}}
		+\sum_{{k_1<k_{r+1},\,k_1>\cdots >k_r\atop k_{r+1}>\cdots>k_{r+s}}}
		\frac{q^{k_{r+1}}+q^{k_{r+1}}(1-q^{k_1})}{(1-q^{k_1})^{n_1}\cdots(1-q^{k_{r+s}})^{n_{r+s}}}\\
	&&\hskip 10mm+\sum_{{k_1=k_{r+1},\,k_1>\cdots >k_r\atop k_{r+1}>\cdots>k_{r+s}}}
	\frac{q^{k_1}-q^{k_1}(1-q^{k_{1})}}{(1-q^{k_1})^{n_1}\cdots(1-q^{k_{r+s}})^{n_{r+s}}}.\\
	&=&\bar{\mathfrak z}_q^{\sqshu}\Big(z_{n_1}\big(z_{n_2}\cdots z_{n_r}*T(z_{n_{r+1}}\cdots z_{n_{r+s}})\big)
	+z_{n_{r+1}}\big(T(z_{n_1}\cdots z_{n_r})*z_{n_{r+2}}\cdots z_{n_{r+s}})\big)\\
	&&\hskip 60mm+(z_{n_1+n_{r+1}}-z_{n_1+n_{r+1}-1})(z_{n_2}\cdots z_{n_r}*z_{n_{r+2}}\cdots z_{n_{r+s}})\Big)\\
	&=&\bar{\mathfrak z}_q^{\sqshu}(z_{n_1}\cdots z_{n_r}\qshu z_{n_{r+1}}\cdots z_{n_{r+s}}),
\end{eqnarray*}}
which proves the claim.
\end{proof}


\subsection{The differential algebra structure}
\label{ssect:diffalg}

We introduce the derivation $\delta:=q\frac{d}{dq}$. For this it will be convenient to use $q$MZVs with arguments which can be zero. Recall that for example:
\allowdisplaybreaks{
\begin{eqnarray*}
	\bar{\mathfrak{z}}_q(0)	&=&\sum_{m>0}q^m=\frac{q}{1-q},\\
	\bar{\mathfrak{z}}_q(0,0)	&=&\sum_{m_1>m_2>0}q^{m_1}
						=\sum_{m>0}(m-1)q^m=\left(\frac{q}{1-q}\right)^2,\\
	\bar{\mathfrak{z}}_q(a,0)	&=&\sum_{m_1>m_2>0}\frac{q^{m_1}}{(1-q^{m_1})^a}
 						=\sum_{m>0}\frac{(m-1)q^{m}}{(1-q^{m_1})^a}.
\end{eqnarray*}}

\begin{prop}
\label{prop:q-deriv}
For any $a \in \mathbb{Z}$ we have:
\begin{equation}
\label{deriv-formula}
	\delta\bar{\mathfrak{z}}_q(a)=(1-a)\big(\bar{\mathfrak{z}}_q(a,0)+\bar{\mathfrak{z}}_q(a)\big)
							+a\big(\bar{\mathfrak{z}}_q(a+1,0)+\bar{\mathfrak{z}}_q(a+1)\big).
\end{equation}
\end{prop}

\begin{proof}
This is a straightforward computation:
\allowdisplaybreaks{
\begin{eqnarray*}
	\delta\bar{\mathfrak{z}}_q(a)	&=&\sum_{m>0}\delta\frac{q^m}{(1-q^m)^a}\\
						&=&\sum_{m>0}\frac{mq^m}{(1-q^m)^a}
							+\sum_{m>0}\frac{amq^{2m}}{(1-q^m)^{a+1}}\\
						&=&\bar{\mathfrak{z}}_q(a,0)+\bar{\mathfrak{z}}_q(a)
							+\sum_{m>0}\frac{-amq^m(1-q^m)}{(1-q^m)^{a+1}}
							+\sum_{m>0}\frac{amq^m}{(1-q^m)^{a+1}}\\
						&=&\bar{\mathfrak{z}}_q(a,0)
							+\bar{\mathfrak{z}}_q(a)
							-a\big(\bar{\mathfrak{z}}_q(a,0)
							+\bar{\mathfrak{z}}_q(a)\big)
							+a\big(\bar{\mathfrak{z}}_q(a+1,0)+\bar{\mathfrak{z}}_q(a+1)\big)\\
						&=&(1-a)\big(\bar{\mathfrak{z}}_q(a,0)+\bar{\mathfrak{z}}_q(a)\big)
							+a\big(\bar{\mathfrak{z}}_q(a+1,0)+\bar{\mathfrak{z}}_q(a+1)\big).
\end{eqnarray*}}
\end{proof}

\begin{exam}\label{exam:deriv}
{\rm{For later use we display:
\begin{eqnarray*}
	\delta\bar{\mathfrak{z}}_q(2) &=&-\big(\bar{\mathfrak{z}}_q(2,0)+\bar{\mathfrak{z}}_q(2)\big)
								+2\big(\bar{\mathfrak{z}}_q(3,0)+\bar{\mathfrak{z}}_q(3)\big),\\
	\delta\bar{\mathfrak{z}}_q(1) &=& \bar{\mathfrak{z}}_q(2,0)+\bar{\mathfrak{z}}_q(2),\\
	\delta\bar{\mathfrak{z}}_q(0) &=&\bar{\mathfrak{z}}_q(0,0)+\bar{\mathfrak{z}}_q(0),\\
	\delta\bar{\mathfrak{z}}_q(-1) &=& 2 \big(\bar{\mathfrak{z}}_q(-1,0)+\bar{\mathfrak{z}}_q(-1)\big)
					-\big(\bar{\mathfrak{z}}_q(0,0)+\bar{\mathfrak{z}}_q(0)\big),\\
	\delta\bar{\mathfrak{z}}_q(-2) &=& 3 \big(\bar{\mathfrak{z}}_q(-2,0)+\bar{\mathfrak{z}}_q(-2)\big)
					-2\big(\bar{\mathfrak{z}}_q(-1,0)+\bar{\mathfrak{z}}_q(-1)\big).
\end{eqnarray*}
Note that $\bar{\mathfrak{z}}_q(1,0)+\bar{\mathfrak{z}}_q(1)$ does not appear.}}
\end{exam}
\goodbreak
\begin{prop}
\label{prop:derivations}
For any integer $a\ge 2$ we have:
\begin{equation}
\label{deriv1}
	\bar{\mathfrak{z}}_q(a,0)+\bar{\mathfrak{z}}_q(a)
	=\frac{1}{a-1}\left(\sum_{j=2}^{a-1} \delta \bar{\mathfrak{z}}_q(j)+\bar{\mathfrak{z}}_q(2,0)
			+\bar{\mathfrak{z}}_q(2)\right).
\end{equation}
For $a\ge 0$ we have:
\begin{equation}
\label{deriv2}
	\bar{\mathfrak{z}}_q(-a,0)+\bar{\mathfrak{z}}_q(-a)
	=\frac{1}{a+1}\left(\sum_{j=1}^a \delta \bar{\mathfrak{z}}_q(-j)+\bar{\mathfrak{z}}_q(0,0)
				+\bar{\mathfrak{z}}_q(0)\right).
\end{equation}
\end{prop}

\begin{proof}
According to Example \ref{exam:deriv}, first (resp. second) assertion is valid for $a=2$ (resp. $a=0$). Note that the sums involving $\delta$-terms are empty in this case. From \eqref{deriv-formula} we get:
\begin{equation}\label{rec-delta}
\bar{\mathfrak{z}}_q(a+1,0)+\bar{\mathfrak{z}}_q(a+1)=\frac{1}{a}\delta\bar{\mathfrak{z}}_q(a)
+\frac{a-1}{a}\big(\bar{\mathfrak{z}}_q(a,0)+\bar{\mathfrak{z}}_q(a)\big).
\end{equation}
Plugging first assertion for $a$ inside \eqref{rec-delta} returns it for $a+1$. Second assertion is treated similarly, which proves Proposition \ref {prop:derivations} by induction on $a$. Note that $\bar{\mathfrak{z}}_q(1,0)$ and $\bar{\mathfrak{z}}_q(1)$ do not enter into this game.
\end{proof}
\vskip 3mm
\noindent Proposition \ref{prop:q-deriv} generalizes to any $q$MZV:
\begin{prop}
\label{prop:q-deriv-Gen}
For any $a_1,\ldots,a_k \in \mathbb{Z}^k$ we have:
\allowdisplaybreaks{
\begin{eqnarray*}
	\delta\bar{\mathfrak{z}}_q(a_1,\ldots,a_k)
		&=&\left(k-\sum_{r=1}^k a_r(k-r+1)\right)\bar{\mathfrak{z}}_q(a_1,\ldots,a_k)\\
	& &	+\sum_{r=1}^k a_r(k-r+1)\bar{\mathfrak{z}}_q(a_1,\ldots,a_{r}+1,\ldots,a_k)\\
	&& +\sum_{s=1}^k\left(1-\sum_{r=1}^s a_r\right)\bar{\mathfrak{z}}_q(a_1,\ldots,a_s,0,a_{s+1},\ldots,a_k) \\
	&&+\sum_{1\le r\le s\le k}a_r \bar{\mathfrak{z}}_q(a_1,\ldots,a_r+1,\ldots,a_s,0,a_{s+1},\ldots,a_k).
\end{eqnarray*}}
\end{prop}

\begin{proof}
The computation is more involved now:
\allowdisplaybreaks{
\begin{eqnarray*}
	\lefteqn{\delta\bar{\mathfrak{z}}_q(a_1,\ldots,a_k)
	=\sum_{m_1>\cdots >m_k>0}\delta\frac{q^{m_1}}{(1-q^{m_1})^{a_1}\cdots(1-q^{m_k})^{a_k}}}\\
	&=&\sum_{m_1>\cdots >m_k>0}\frac{m_1q^{m_1}}{(1-q^{m_1})^{a_1}\cdots(1-q^{m_k})^{a_k}}\\
	& &	+\sum_{m_1>\cdots >m_k>0}\sum_{r=1}^k\frac{a_rm_rq^{m_1+m_r}}
		{(1-q^{m_1})^{a_1}\cdots(1-q^{m_r})^{a_r+1}\cdots(1-q^{m_k})^{a_k}}\\
	&=&\sum_{m_1>\cdots >m_k>0}\frac{m_1q^{m_1}}{(1-q^{m_1})^{a_1}\cdots(1-q^{m_k})^{a_k}}\\
	& &	+\sum_{m_1>\cdots >m_k>0}\sum_{r=1}^k
			\frac{a_rm_rq^{m_1}-a_rm_rq^{m_1}(1-q^{m_r})}
			{(1-q^{m_1})^{a_1}\cdots(1-q^{m_r})^{a_r+1}\cdots(1-q^{m_k})^{a_k}}.
\end{eqnarray*}}
We decompose the integers $m_r$ as:
\begin{equation*}
	m_r=(k-r+1)+(m_r-m_{r+1}-1)+\cdots+(m_{k-1}-m_k-1)+(m_k-1),
\end{equation*}
which gives:
\goodbreak
\allowdisplaybreaks{
\begin{eqnarray*}
	\lefteqn{\delta\bar{\mathfrak{z}}_q(a_1,\ldots,a_k)}\\
	&=&k\bar{\mathfrak{z}}_q(a_1,\ldots,a_k)
		+\sum_{s=1}^k\bar{\mathfrak{z}}_q(a_1,\ldots,a_s,0,a_{s+1},\ldots,a_k)\\
	&& + \sum_{r=1}^k a_r\left((k-r+1)\bar{\mathfrak{z}}_q(a_1,\ldots,a_r+1,\ldots, a_k)
		+\sum_{s=r}^k\bar{\mathfrak{z}}_q(a_1,\ldots a_r+1,\ldots,a_s,0,a_{s+1},\ldots,a_k)\right)\\
	&& - \sum_{r=1}^k a_r\left((k-r+1)\bar{\mathfrak{z}}_q(a_1,\ldots,a_r,\ldots, a_k)
		+\sum_{s=r}^k\bar{\mathfrak{z}}_q(a_1,\ldots,a_r,\ldots,a_s,0,a_{s+1},\ldots,a_k)\right)\\
	&=&\left((k-\sum_{r=1}^k a_r(k-r+1)\right)\bar{\mathfrak{z}}_q(a_1,\ldots,a_k)
		+\sum_{r=1}^k a_r(k-r+1)\bar{\mathfrak{z}}_q(a_1,\ldots,a_{r+1},\ldots,a_k)\\
	&& + \sum_{s=1}^k\left(1-\sum_{r=1}^s a_r\right)\bar{\mathfrak{z}}_q(a_1,\ldots,a_s,0,a_{s+1},\ldots,a_k)\\
	& &	+\sum_{1\le r\le s\le k}a_r \bar{\mathfrak{z}}_q(a_1,\ldots,a_r+1,\ldots,a_s,0,a_{s+1},\ldots,a_k)
\end{eqnarray*}}
\end{proof}

With Proposition \ref{prop:derivations} at hand let us return to (\ref{DP-Euler-slim}) and (\ref{DD-Euler-slim}), i.e., the $q$-generalizations of Euler's decomposition formula for negative values. First, observe that from (\ref{deriv2}) it follows that:
$$
	\bar{\mathfrak{z}}_q(-1,0) = \frac{1}{2} \delta\bar{\mathfrak{z}}_q(-1)  -  \bar{\mathfrak{z}}_q(-1)
					+  \frac{1}{2}\big(\bar{\mathfrak{z}}_q(0,0)+ \bar{\mathfrak{z}}_q(0)\big).
$$

As a more involved example, we consider the product $\bar{\mathfrak{z}}_q(-2)\bar{\mathfrak{z}}_q(-2)$. The last term on the right hand side of (\ref{DP-Euler-slim}) yields:
$$
	 \sum_{j=1}^{2} (-1)^j  {3-j \choose j-1,2-j,2-j} \bar{\mathfrak{z}}_q(-j,0) 
	 		= - 2 \bar{\mathfrak{z}}_q(-1,0) + \bar{\mathfrak{z}}_q(-2,0).
$$ 
Returning to Example \ref{exam:deriv}, we see that:
$$
	- 2 \bar{\mathfrak{z}}_q(-1,0) + \bar{\mathfrak{z}}_q(-2,0) 
	= -\frac{2}{3} \delta\bar{\mathfrak{z}}_q(-1) + \frac{1}{3}\delta\bar{\mathfrak{z}}_q(-2) 
	-  \bar{\mathfrak{z}}_q(-2)
	+ 2 \bar{\mathfrak{z}}_q(-1) 
	-  \frac{2}{3}\big(\bar{\mathfrak{z}}_q(0,0)+\bar{\mathfrak{z}}_q(0)\big).
$$

From this we deduce a formula expressing linear combinations of $ \bar{\mathfrak{z}}_q(-j,0)$ in terms of $ \delta\bar{\mathfrak{z}}_q(-i)$, $\bar{\mathfrak{z}}_q(-k)$ and $\bar{\mathfrak{z}}_q(0,0)+\bar{\mathfrak{z}}_q(0)$.

\begin{cor}
Let $\kappa_i$ be coefficients in $\mathbb{Q}$.   
\label{cor:deriv-terms}
$$
	\sum_{i=1}^n \kappa_i \bar{\mathfrak{z}}_q(-i,0) 
	=\sum_{i=1}^n  \left(\sum_{j=1}^{n-i+1} \frac{\kappa_{n+1-j}}{n+2-j}\right)\delta \bar{\mathfrak{z}}_q(-i) 
	- \sum_{i=1}^n \kappa_i\bar{\mathfrak{z}}_q(-i) 
		+ \Big(\sum_{i=1}^n \frac{\kappa_i}{i+1}\Big)(\bar{\mathfrak{z}}_q(0,0)
				+\bar{\mathfrak{z}}_q(0)).
$$
\end{cor}

\section{Double $q$-shuffle relations}\label{sect:doubleqshuf}
\subsection{Double $q$-shuffle relations for modified qMZVs}

\noindent We can define a bijective map that changes the letter $z_n$ into the word $p^{n}y$:
\allowdisplaybreaks{
\begin{eqnarray*}
	{\frak r}: \wt Y^*&\tilde{\longrightarrow} & W \\
	z_{n_1}\cdots z_{n_k} &\longmapsto & p^{n_1}y\cdots  p^{n_k}y,
\end{eqnarray*}}
Above we have seen that:
\begin{equation*}
	\bar{\mathfrak{z}}_q(n_1,\ldots,n_k)	= \bar{\mathfrak{z}}_q^{\sshu}(p^{n_1}y\cdots p^{n_k}y)
								= \bar{\mathfrak{z}}_q^{\sqshu}(z_{n_1}\cdots z_{n_k}).
\end{equation*}
From this we obtain the double $q$-shuffle relations: 
\begin{eqnarray}\label{eq:qds}
	\bar{\mathfrak{z}}_q^{\sqshu}&=&\bar{\mathfrak{z}}_q^{\sshu}\circ{\frak r},\notag\\
	\bar{\mathfrak{z}}_q^{\sshu}(u)\bar{\mathfrak{z}}_q^{\sshu}(v)	 &=&\bar{\mathfrak{z}}_q^{\sshu}(u\shu v),\\
	\bar{\mathfrak{z}}_q^{\sqshu}(u')\bar{\mathfrak{z}}_q^{\sqshu}(v')&=&\bar{\mathfrak{z}}_q^{\sqshu}(u'\qshu v')\notag
\end{eqnarray}
for any words $u,v\in W$, respectively~ $u',v'\in \wt Y^*$.  From \eqref{eq:qds} we immediately deduce:
\begin{equation}
\label{eq:qds-bis}
	\bar{\mathfrak{z}}_q^{\sshu}\big({\frak r}(u')\shu{\frak r}(v')-{\frak r}(u'\qshu v')\big)=0
\end{equation}
for any $u',v'\in\wt Y^*$, or alternatively:
\begin{equation}
\label{eq:qds-ter}
	\bar{\mathfrak{z}}_q^{\sqshu}\big({\frak r}^{-1}(u)\qshu{\frak r}^{-1}(v)-{\frak r}^{-1}(u\shu v)\big)=0
\end{equation}
for any $u,v\in W$.
\subsection{Double $q$-shuffle relations for non-modified qMZVs}
We introduce a natural notion of weight for words\footnote{This should not bring confusion with the notion of weight for a Rota--Baxter operator introduced before.} both in $\wt Y^*$ and in $W$, which takes integer values of any sign:
\begin{equation}
w(z_{n_1}\cdots z_{n_k})=w(p^{n_1}y\cdots  p^{n_k}y):=n_1+\cdots+n_k.
\end{equation}
Let us denote by $\Cal Y$ (resp. $\Cal W$) the ${\mathbb Q}$-vector space spanned by $\wt Y^*$ (resp. $W$) endowed with the product $\qshu$ (resp. $\shu$). Both products are filtered, but not graded, with respect to the weight: if $\Cal Y^{(n)}$ (resp. $\Cal W^{(n)}$) stands for the linear span of words in $\wt Y^*$ (resp. $W$) of weight $\le n$ we have:
\begin{eqnarray*}
\Cal Y^{(n)}\qshu\Cal Y^{(m)}&\subseteq & \Cal Y^{(n+m)},\\
\Cal W^{(n)}\shu\Cal W^{(m)}&\subseteq & \Cal W^{(n+m)}.
\end{eqnarray*}
A graded version can be introduced as follows: introduce the Laurent polynomials in the indeterminate $h$ with coefficients in $\Cal Y$ (resp. $\Cal W$), and give weight $1$ to $h$. The swap ${\mathfrak r}$ is linearly extended to a linear isomorphism from $\Cal Y$ onto $\Cal W$, and then from $\Cal Y[h^{-1},h]$ onto $\Cal W[h^{-1},h]$ by extension of scalars. The products $\qshu$ and $\shu$ are extended $h$-bilinearly to $\Cal Y$ and $\Cal W$ respectively. The letter $q$ will stand here for $1-h$, for reasons which will become clear in the sequel. Consider the linear transformation:
\begin{align*}
H_q:\Cal Y[h^{-1},h]\ \longrightarrow \ & \Cal Y[h^{-1},h]\\
u'\ \longmapsto \ & h^{w(u')}u',
\end{align*}
and consider the analogous map on $\Cal W[h^{-1},h]$ which will be also called $H_q$. Let us now introduce two products $\qshu\!_q\,$ and $\shu\!_q\,$ on $\Cal Y[h^{-1},h]$ and $\Cal W[h^{-1},h]$ respectively, by means of:
\begin{align}
u'\qshu\!_q\, v':=\ &H_q^{-1}(H_qu'\qshu H_qv'),\\
u\shu\!_q\, v:=\ &H_q^{-1}(H_qu\shu H_qv).
\end{align}
We $h$-bilinearly extend the maps ${\mathfrak z}_q^{\sqshu}$ and ${\mathfrak z}_q^{\sshu}$ to the Laurent polynomials by sending $h$ to $1-q$. We can now display the double $q$-shuffle relations for non-modified $q$-multiple zeta values:
\begin{prop}
\rm{\begin{eqnarray}\label{eq:qds-nm}
	{\mathfrak{z}}_q^{\sqshu}&=&{\mathfrak{z}}_q^{\sshu}\circ{\frak r},\notag\\
	{\mathfrak{z}}_q^{\sshu}(u){\mathfrak{z}}_q^{\sshu}(v)	 &=&{\mathfrak{z}}_q^{\sshu}(u\shu\!_q\, v),\\
	{\mathfrak{z}}_q^{\sqshu}(u'){\mathfrak{z}}_q^{\sqshu}(v')&=&{\mathfrak{z}}_q^{\sqshu}(u'\qshu\!_q\, v')\notag
\end{eqnarray}}
\end{prop}
\begin{proof}
This is immediate from \eqref{modz}, \eqref{weighting}, \eqref{eq:qds} and the definitions of the two new products.
\end{proof}
\noindent Note that the two new products are now graded (with respect to the weight), and we have:
\begin{eqnarray*}
	(yv)\shu\!_q\, u=v\shu\!_q\, (yu)&=&y(v\shu\!_q\, u),						\\
			pv\shu\!_q\, pu&=&p(v\shu\!_q\, pu)+p(pv\shu\!_q\, u) - hp(v\shu\!_q\, u),	\\
			hdv\shu\!_q\, du&=&v\shu\!_q\, du + dv\shu\!_q\, u - d(v\shu\!_q\, u),		\\
	dv\shu\!_q\, pu=pu\shu\!_q\, dv&=&d(v\shu\!_q\, pu)- v\shu\!_q\, u+hdv\shu\!_q\, u, 		
\end{eqnarray*}
as well as:
\begin{equation}\label{qshu-q}
u'\qshu\!_q\, v'=T_q^{-1}(T_qu' *T_qv'),
\end{equation}
where the operator $T_q$ is defined by:
\begin{equation}
T_q(z_nv'):=(z_n-hz_{n-1})v'
\end{equation}
for any $n\in\mathbb Z$ and for any $v'\in\wt Y^*$. Strictly speaking, the operator $T_q^{-1}$ is defined on the space $\Cal Y[h^{-1},h]]$ of Laurent series with coefficients in $\Cal Y$, by the series:
\begin{equation}
T_q^{-1}(z_nv')=\sum_{k\ge 0}h^kz_{n-k}v',
\end{equation}
but it does not show up in the expression of the product $\qshu\!_q\,$ in terms of the ordinary quasi-shuffle product $*$. Indeed \eqref{qshu-q} yields:
\begin{equation}
z_mu'\qshu\!_q\, z_nv'=z_m\big(u'*T_q(z_nv')\big)+z_n\big(T_q(z_mu')*v'\big)+T_q\big(z_{m+n}(u'*v')\big).
\end{equation}
Finally, any Laurent polynomial in $h$ can be seen as a formal series in $q$. All results in this paragraph, except the grading, still hold over the ring $\mathbb Q[[q]]$ with $h=1-q$.
\subsection{Digression on Schlesinger $q$-MZVs}
The Schlesinger multiple zeta values are defined as follows:
\begin{eqnarray}
     \mathfrak{z}_q^S(n_1,\ldots,n_k)	
     			&=&(1-q)^w\underbrace{P_q \: \big[P_q \: [\cdots P_q}_{n_1}\: [\bar{y}
                                     \cdots
                                     \underbrace{P_q \: [P_q\:[ \cdots P_q}_{n_k}\: [\bar{y}]]]]
                                     \cdots ]\big](1)										\nonumber\\
                       	&=& \sum_{m_1 > \dots > m_k > 0}
				\frac{1}{[m_1]_q^{n_1}\cdots [m_k]_q^{n_k}}.  			\label{qMZVs-S}
\end{eqnarray}
They are defined for $|q|>1$ for $n_j\ge 1$ and $n_1\ge 2$, and converge to the corresponding classical MZVs for $q\to 1$ (by the monotone convergence theorem). A straightforward computation yields:
\begin{equation}
\mathfrak{z}_q^S(n_1,\ldots,n_k)=
\sum_{m_1>\cdots >m_k>0}\frac{(q^{-1})^{(m_1-1)n_1+\cdots+ (m_k-1)n_k}}{[m_1]_{q^{-1}}^{n_1}\cdots[m_k]_{q^{-1}}^{n_k}}.
\end{equation}
In other words, $\mathfrak{z}_q^S(n_1,\ldots,n_k)=\mathfrak{z}_{q^{-1}}^B(n_1,\ldots,n_k)$, where the superscript $B$ stands for the Bradley model (see \cite{Bradley1} and \cite{CEM}). Hence $\mathfrak{z}_q^S(n_1,\ldots,n_k)$ makes also sense as a formal series in $q^{-1}$ for $n_j\ge 0$ and $n_1\ge 1$. Now let us introduce the following notations:
\begin{eqnarray*}
\wt Y_1^*&:=&\{z_{n_1}\cdots z_{n_k}\in \wt Y^*,\,n_j\ge 1\},\\
\wt Y_2^*&:=&\{z_{n_1}\cdots z_{n_k}\in \wt Y^*,\,n_j\ge 1\hbox{ and } n_1\ge 2\},\\
\Cal Y_1&:=& \hbox {$\mathbb Q$-linear span of } \wt Y_1^*,\\
\Cal Y_2&:=& \hbox {$\mathbb Q$-linear span of } \wt Y_2^*,\\
 \mathfrak{z}_q^{S\sqshu}(z_{n_1}\cdots z_{n_k})&:=&	 \mathfrak{z}_q^S(n_1,\ldots,n_k),\, 	z_{n_1}\cdots z_{n_k}\in \wt Y_1^*.
\end{eqnarray*}
\begin{prop}
The following diagram commutes:
\rm{
\diagramme{
\xymatrix{
\Cal Y_1[[q^{-1}]]\ar[dr]^{\mathfrak{z}_q^{S\sqshu}}&\\
q^{-1}\Cal Y_2[[q^{-1}]]\ar[u]_{T_q} \ar[r]_{\mathfrak{z}_q^{\sqshu}}&\mathbb Q[[q^{-1}]]
}
}}
\end{prop}
\begin{proof}
This is a straightforward formal computation using the obvious equality:
$$1-(1-q)[m_1]_q=q^{m_1}.$$
\end{proof}
\ignore{
\redtext{Note that  $\mathfrak{z}_q^{\sqshu}$ is multiplicative for the product $\qshu_q$ whereas $ \mathfrak{z}_q^{S\sqshu}$ is multiplicative for the ordinary quasi-shuffle product $*$. Hence it is tempting to transport the deconcatenation coproduct by means of $T_q^{-1}$, thus grasping a Hopf algebra structure including the product $\qshu_q$... But this is only formal in $q$. We manipulate here series in both $q$ and $q^{-1}$ unlimited in both directions, which is desperate with respect to any reasonable algebraic structure.}
}
\subsection{$q$-analog of regularized double $q$-shuffle relations}
\label{ssect:reg-doubleqshuf}
In the following we restrict to non-negative indices. Thanks to the $q$-regularization, we observe that the double $q$-shuffle relation yields:
\allowdisplaybreaks{
\begin{eqnarray*}
	 \bar{\mathfrak{z}}_q(1) \bar{\mathfrak{z}}_q(2) &=&  \bar{\mathfrak{z}}_q(1,2) + \bar{\mathfrak{z}}_q(2,1) + \bar{\mathfrak{z}}_q(3)
	  - \bar{\mathfrak{z}}_q(1,1) - \bar{\mathfrak{z}}_q(2,0) - \bar{\mathfrak{z}}_q(2)\\
	 &=& \bar{\mathfrak{z}}_q(1,2) + 2 \bar{\mathfrak{z}}_q(2,1) - \bar{\mathfrak{z}}_q(2,0) - \bar{\mathfrak{z}}_q(1,1).
\end{eqnarray*}}
From which we obtain the relation:
\begin{equation}\label{qEuler}
	\bar{\mathfrak{z}}_q(3) - \bar{\mathfrak{z}}_q(2) = \bar{\mathfrak{z}}_q(2,1).
\end{equation}
For the proper $q$MZVs defined in (\ref{2qMZVs}) this gives: 
$$
	\mathfrak{z}_q(3) - (1-q)\mathfrak{z}_q(2) = \mathfrak{z}_q(2,1),
$$
which in the limit where $q$ goes to one, reduces to (\ref{euler}), i.e., the classical relation $\zeta(2,1)=\zeta(3)$.  Interestingly enough, Equation \eqref{qEuler} is equivalent to an algebraic identity established by E. T. Bell in the Thirties of the last century (see \cite{B36}, Page 158 thereof)\footnote{We thank W. Zudilin for kindly drawing our attention to reference \cite{B36}.}.

In the following we denote by $\mathbb{Q}. \wt Y^*_{\smop{conv}}$ the free associative algebra generated by the set of words $\wt Y^*_{\smop{conv}}$, the submonoid of words $z_{n_1}\cdots z_{n_k}$ with letters in $\wt Y_+:=\{z_n,\,n\in\mathbb N\}$, and $n_1 >1$. We arrive at a $q$-analog of Hoffman's regularization relations. By virtue of the $q$-regularization, this appears to be a particular case of \eqref{eq:qds-bis}:
\begin{prop}
\label{thm:q-Hoffman}
For any $v \in  \mathbb{Q}. \wt Y^*_{\smop{conv}}$ we have:
{\rm
\begin{equation}
\label{q-reg1}
	\bar{\mathfrak{z}}_q^{\sshu}\big(py \shu {\frak r}(v)-{\frak r}(z_1\qshu v)\big)=0,
\end{equation}
respectively
\begin{equation}
\label{q-reg2}
	\mathfrak{z}_q^{\sshu}\big(py \shu {\frak r}(v)-{\frak r}(z_1\qshu v)\big)=0.
\end{equation}
}
\end{prop}

Since terms of depth smaller than $|v|+1$ disappear in the limit $q \to 1$, identity (\ref{q-reg2}) reduces to Hoffman's regularization relations (\ref{reg}) for MZVs.\\

As there are no regularization issues involved, no correction analogous to $\rho$ (with the notations of Section \ref{sect:algMZVs}) is needed to go from the $q$-quasi-shuffle picture to the $q$-shuffle picture or vice-versa. The first nontrivial manifestation of the correction $\rho$ for classical multiple zeta values appears when computing $\zeta^{\sqshu}(1)\zeta^{\sqshu}(1)$ and $\zeta^{\sshu}(1)\zeta^{\sshu}(1)$. Using quasi-shuffle and shuffle relations respectively we get:
\allowdisplaybreaks{
\begin{eqnarray*}
	\theta^2=\zeta^{\sqshu}(1)\zeta^{\sqshu}(1)	&=&2\zeta^{\sqshu}(1,1)+\zeta(2),\\
	\theta^2=\zeta^{\sshu}(1)\zeta^{\sshu}(1)		&=&2\zeta^{\sshu}(1,1),
\end{eqnarray*}}
hence:
\begin{equation}
\label{one-one}
	\zeta^{\sshu}(1,1)-\zeta^{\sqshu}(1,1)=\frac 12\zeta(2).
\end{equation}
It is instructive to compute the $q$-analog with \eqref{one-one} in mind, using $q$-quasi-shuffle and $q$-shuffle relations respectively:
\begin{eqnarray*}
	\bar{\mathfrak{z}}_q(1)\bar{\mathfrak{z}}_q(1)	&=&2\bar{\mathfrak{z}}_q(1,1)+\bar{\mathfrak{z}}_q(2)-2\bar{\mathfrak{z}}_q(1,0)
												-\bar{\mathfrak{z}}_q(1),\\
										&=&2\bar{\mathfrak{z}}_q(1,1)-\bar{\mathfrak{z}}_q(1,0).
\end{eqnarray*}
Hence we get:
\begin{equation*}
	\bar{\mathfrak{z}}_q(2)=\bar{\mathfrak{z}}_q(1,0)+\bar{\mathfrak{z}}_q(1),
\end{equation*}
which in turn gives:
\begin{equation*}
	{\mathfrak{z}}_q(2)=(1-q)\big({\mathfrak{z}}_q(1,0)+{\mathfrak{z}}_q(1)\big).
\end{equation*}
Considering the limit for $q\to 1$, we get:
\begin{equation*}
	\mopl{lim}_{q\to 1}(1-q)\big({\mathfrak{z}}_q(1,0)+{\mathfrak{z}}_q(1)\big)=\frac{\pi^2}{6}.
\end{equation*}


\subsection{Double $q$-shuffle relation and $\delta$-derivation terms}
\label{ssect:deriv-doubleqshuf}

First observe that:
\allowdisplaybreaks{
\begin{equation*}
	 \bar{\mathfrak{z}}_q(2) \bar{\mathfrak{z}}_q(b) =  \bar{\mathfrak{z}}_q(2,b) + \bar{\mathfrak{z}}_q(b,2) + \bar{\mathfrak{z}}_q(b+2)
	  - \bar{\mathfrak{z}}_q(b,1) - \bar{\mathfrak{z}}_q(2,b-1) - \bar{\mathfrak{z}}_q(b+1).
\end{equation*}}
Next, we deduce from:
\allowdisplaybreaks{
\begin{eqnarray*}
	\bar{\mathfrak{z}}_q(2)\bar{\mathfrak{z}}_q(b) &=&
	\sum_{l=0}^{1}\ \sum_{k=0}^{1-l}  
		(-1)^k {l+b-1 \choose b-1}{b \choose k} \bar{\mathfrak{z}}_q(b+l,2-l-k) 						\\
	& & \hspace{0.5cm}
	 	+ \sum_{l=0}^{b-1}\ \sum_{k=0}^{\min(2,b-1-l)}  
		(-1)^k {l+1 \choose 1}{2 \choose k} \bar{\mathfrak{z}}_q(2+l,b-l-k) 				\nonumber\\
	& & \hspace{1cm}
		- \sum_{k=1}^{2} \beta_{2-k} \bar{\mathfrak{z}}_q(2+b-k)
			+ \sum_{j=1}^{1} \alpha_{1+b-j} \delta \bar{\mathfrak{z}}_q(1+b-j),		
\end{eqnarray*}}
that for any $b>2$, and (\ref{coeff}):
\allowdisplaybreaks{
\begin{eqnarray*}
	 \delta\bar{\mathfrak{z}}_q(b) 
	&=& (b+1)\bar{\mathfrak{z}}_q(b+1) -  (b-1) \bar{\mathfrak{z}}_q(b) - \bar{\mathfrak{z}}_q(b+2)\\
	& & +2b\bar{\mathfrak{z}}_q(b+1,1) + (1-b) \bar{\mathfrak{z}}_q(b,1) - \bar{\mathfrak{z}}_q(2,b-1) + \bar{\mathfrak{z}}_q(2,b-2)\\
	& & \hspace{0.5cm}
	 	+ \sum_{l=1}^{b-2}\ \sum_{k=0}^{\min(2,b-1-l)}  
		(-1)^k {l+1 \choose 1}{2 \choose k} \bar{\mathfrak{z}}_q(2+l,b-l-k) .				\nonumber
\end{eqnarray*}}


\section*{Concluding remarks}

In this work we have further developed the results from our previous work \cite{CEM}. The main goal was to unfold the double $q$-shuffle structure underlying the $q$MZVs proposed by Y.~Ohno. J.~Okuda and W.~Zudilin in \cite{OhOkZu}. The main observation is that these $q$MZVs can be considered with integer arguments of any sign --the $q$-parameter figures as a proper regulator--, and that the double $q$-shuffle structure makes sense in this extended setting.\\

However, we have not touched upon two important questions. First, we did not find a proper bi- or Hopf algebra for the $q$-shuffle. Note that the $q$-quasi-shuffle we found can be reduced to the usual quasi-shuffle product described by Hoffman, and as such comes with the common deconcatenation coproduct. As a remark, we may also add that we are still trying to find a proper combinatorial interpretation of the $q$-shuffle product. Second, the extension to negative arguments has no good limit for $q$ going to one. We are lacking a proper renormalization procedure, that would allow us to give meaning to the $q \to 1$ limit outside the usual convergence constraints for classical MZVs. Both these questions are to be addressed in a forthcoming work.


\end{document}